\documentclass[11pt,a4paper]{article}

\usepackage{amsmath}
\usepackage{amsfonts}
\usepackage{amssymb}
\usepackage{amsthm,bbm}
\usepackage{a4wide}
\usepackage[]{amsmath,amssymb,amsfonts,latexsym,amsthm,enumerate,fullpage}
\newcommand{\R}{\mathbbm{R}}

\numberwithin{equation}{section}

\newtheorem{thm}{Theorem}[section]
\newtheorem{pro}[thm]{Proposition}
\newtheorem{lem}[thm]{Lemma}
\newtheorem{cor}[thm]{Corollary}

\newtheorem{que}[thm]{Question}
\theoremstyle{Definition}
\newtheorem{dfn}[thm]{Definition}

\newtheorem{rem}[thm]{Remark}
\newtheorem{exm}[thm]{Example}
\theoremstyle{plain}
\begin{document}

\title{Bounded Degree Cosystolic Expanders of Every Dimension}
\author{
Shai Evra \thanks{Hebrew University, ISRAEL. Email: 		\texttt{shai.evra@mail.huji.ac.il}.
Research supported in part by the ERC.}
\and Tali Kaufman \thanks{Bar-Ilan University, ISRAEL. Email: \texttt{kaufmant@mit.edu}.
Research supported in part by the IRG, ERC and BSF.}
}
\maketitle
\begin{abstract}
In this work we present a new local to global criterion for proving a form of high dimensional expansion, which we term cosystolic expansion.
Applying this criterion on Ramanujan complexes, yields for every dimension, an infinite family of bounded degree complexes with the topological overlapping property.
This answer affirmatively an open question raised by Gromov.
\end{abstract}

%%%%%%%%%%%%%%%%%%%%%%%%%%%%%%%%%%%%%%%%%%%%%%%%%%%%%%%%%%%%%%%%%%%%%%%%%%%%%%%%%%%%%%%%%%%%%%%%%%%
%%%%%%%%%%%%%%%%%%%%%%%%%%%%%%%%%%%%%%%%%%%%%%%%%%%%%%%%%%%%%%%%%%%%%%%%%%%%%%%%%%%%%%%%%%%%%%%%%%%
\section{Introduction}

%story of expanders
Expander graphs have been central objects of study in both computer science and mathematics during the past few decades. Informally, expanders are sparse and highly connected graphs and as such have numerous applications, see \cite{HLW} and \cite{Lub1} and the references therein.
In recent years, a new theory of high dimensional expanders has emerged,
pioneered by the works of Linial-Meshulam-Wallach \cite{LinM},\cite{MW}, and Gromov \cite{Gro} (for a recent survey see \cite{Lub2}).
Linial-Meshulam-Wallach and Gromov suggested two, essentially equivalent, generalizations of the notion of graph expansion to higher dimensions 
(even though both works were in completely different research directions).

%%%%%%%%%%%%%%%%%%%%%%%%%%%%%%%%%%%%%%%%%%%%%%%%%%
\subsection{Gromov's question}
In \cite{Gro}, Gromov studied the complexity of embedding simplicial complexes into Euclidean spaces.
More specifically, he considered the following topological overlapping property:

\begin{dfn}
A $d$-dimensional simplicial complex $X$ has the $c$-{\em topological overlapping property}, for some $c>0$, 
if for every {\em continuous} map $f \colon X \rightarrow \R^d$, there exists a point $p\in \R^d$, whose preimage, $f^{-1}(p)$, 
intersects at least $c$-fraction of the $d$-simplices of $X$, i.e.,
\begin{equation}
|\{\sigma\in X(d)\colon p\in f(\sigma)\}| \geq c \cdot |X(d)|.
\end{equation}
A family of $d$-dimensional simplicial complexes is said to have the {\em topological overlapping property}, 
%or said to be a family of {\em topological expanders},
if there exists $c > 0$, such that each member of the family has the $c$-topological overlapping property.
\end{dfn}

Gromov then proved the remarkable result that, for a fixed $d \in \mathbb{N}$, the family of complete $d$-dimensional complexes have the topological overlapping property.
This result is a very striking generalization of classical results from convex combinatorics due to Boros-Furedi \cite{BF} (for $d=2$) and Barany \cite{Bar} (for $d\geq 3$).
Gromov then proceeded to give other examples of families of complexes which posses the topological overlapping property (e.g. spherical buildings and random complexes).
However, all the examples he provided were of unbounded degree, i.e. the number of faces incident to a single vertex grows with the number of vertices in the complex.
This naturally led Gromov to the following question:

\begin{que}[Gromov] \label{gromov}
Is there an infinite family of bounded degree $d$-dimensional simplicial complexes with the topological overlapping property, for arbitrary $d\geq 2$?
\end{que}

%An analogue question regarding the geometrical overlapping property, namely where continuos maps are replaced with affine ones, was also raised in \cite{Gro}. This question was resolved in the work of \cite{FGLNP}, who gave both random and explicit constructions. Where the explicit construction which \cite{FGLNP} gave were the Ramanujan complexes of \cite{LSV2} (see \cite{Evr} for a generalization).

In a recent work \cite{KKL}, Kaufman, Kazhdan and Lubotzky were able to give an affirmative answer to Gromov's question in dimension two,
i.e., they showed that there exists an infinite family of $2$-dimensional, bounded degree complexes with the topological overlapping property, see \cite[Theorem~1.3]{KKL}.
However, their proof method is inherently suited for dimension two.

In this work we give a complete answer to Gromov's question.

\begin{thm} \label{thm-main-1}
For every $d \in \mathbb{N}$, there exist an infinite family of $d$-dimensional bounded degree complexes with the topological overlapping property.
\end{thm}

An immediate application of Theorem \ref{thm-main-1}, is an improvement of a result by Gromov and Guth \cite[Theorem~2.2]{GG}, 
which gives a high dimensional generalization to a classical result of Kolmogorov and Barzdin \cite{BK}. %(a work which essentially discovered the notion of graph expansion, a few years prior to Pinsker \cite{Pin}).

\begin{cor} \label{cor-gromov-guth}
For every $D\geq 2d+1$, there exists $C=C(D)>0$ and an infinite family of $d$-dimensional bounded degree complexes, $\{X_n\}_n$, which satisfies the following:
For any embedding of $X_n$ into $\mathbb{R}^D$, such that the images of any two non-adjoint simplices of $X_n$ are of distance at least $1$, the volume of the $1$-neighborhood of the image of $X_n$ is at least $C\cdot |X_n|^{\frac{1}{D-d}}$.
\end{cor}

%%%%%%%%%%%%%%%%%%%%%%%%%%%%%%%%%%%%%%%%%%%%%%%%%%
\subsection{A criterion for the topological overlapping property}
Let us start by addressing the question: How can one prove that a certain complex posses a topological overlapping property?
In \cite{Gro}, Gromov showed that the topological overlapping property is implied by a certain notion of high-dimensional expansion, which we now wish to define.
To do so, we first need some notations.

\begin{dfn}
Let $X$ be a $d$-dimensional simplicial complex. For any $-1\leq k \leq d$, define:
\begin{itemize}
\item $X(k) = \{\sigma \in X \; |\; |\sigma|=k+1\}$ - the collection of $k$-dimensional faces of $X$.
\item $C^k=C^k(X;\mathbb{F}_2)=\{f\colon X(k) \rightarrow \mathbb{F}_2\}$ - the space of $k$-dimensional $\mathbb{F}_2$-cochains of $X$.
\item $\delta^k \colon C^k \rightarrow C^{k+1}$, $\delta f(\tau)= \sum_{\sigma\subset \tau} f(\sigma)$ - the $k$-dimensional $\mathbb{F}_2$-coboundary map of $X$.
\item $Z^k = Z^k(X;\mathbb{F}_2) = \ker(\delta^k)$ - the space of $k$-dimensional $\mathbb{F}_2$-cocycles of $X$.
\item $B^k = B^k(X;\mathbb{F}_2) = Im(\delta^{k-1})$ - the space of  $k$-dimensional $\mathbb{F}_2$-coboundaries of $X$.
\item $\|\cdot\| \colon C^k \rightarrow [0,1]$, 
$\|f\|= \sum_{\sigma \in X(k)} \frac{|\{\tau\in X(d) \;|\; \sigma \subset \tau\}|}{{d+1 \choose k+1}\cdot |X(d)|} \cdot f(\sigma)$ - a $k$-dimensional norm on $C^k$.
\end{itemize}
\end{dfn}

We are now in a position to present our first notion of high dimensional expansion.

\begin{dfn}[Coboundary expanders]
Let $X$ be a $d$-dimensional simplicial complex, and let $\epsilon>0$.
We say that $X$ is an $\epsilon$-coboundary expander, if for every $ k =0,1,\ldots,d-1$, 
\begin{equation}
Exp_b^k(X) := \min \left\lbrace \frac{\|\delta(f)\|}{\min_{b\in B^k}\|f+b\|}\; | \; f \in C^k \setminus B^k \right\rbrace \geq \epsilon.
\end{equation}
\end{dfn}

The notion of coboundary expansion first appeared implicitly in the work \cite{LinM}, whose motivation was to generalize to higher dimensional complexes, the phase transition phenomenon of connectivity of random graphs in the Erdos-Reyni model. 
Later on, Gromov \cite{Gro}, came across essentially the same notion of expansion, while studying fiberwise complexetiy, and showing that coboundary expansion implies the topological overlapping property. 

Unfortunately, up to this date, no bounded degree coboundary expanders are known to exists, so one cannot use Gromov's result, to answer his question.
The problem is that coboundary expansion requires vanishing of $\mathbb{F}_2$-cohomology of the complex, a property which is so strong that even certain Ramanujan complexes are known to not satisfy it (see \cite{KKL}). 

In order to bypass this problem, we define a weaker notion of expansion, which we call cosystolic expansion. This notion is strictly weaker than coboundary expansion, since it allows the existence of cocycles which are not coboundaries, as long as they are sufficiently large.

\begin{dfn}[Cosystolic expanders]
Let $X$ be a $d$-dimensional simplicial complex, and let $\epsilon,\mu>0$.
We say that $X$ is an $(\epsilon,\mu)$-cosystolic expander, if for every $ k =0,1,\ldots,d-1$,
\begin{equation}
Exp_z^k(X) := \min \left\lbrace \frac{\|\delta(f)\|}{\min_{z\in Z^k}\|f+z\|}\; | \; f \in C^k \setminus Z^k \right\rbrace \geq \epsilon
\end{equation}
and 
\begin{equation}
Syst^k(X) := \min \{\|z\|\; | \; z \in Z^k \setminus B^k \} \geq \mu.
\end{equation}
\end{dfn}

In \cite{DKW}, the authors strengthen the aforementioned result of Gromov, proving that cosystolic expansion implies the topological overlapping property.

\begin{thm}[Gromov's Topological Overlapping Criterion \cite{DKW}] \label{thm-Gromov}
For every $d\in \mathbb{N}$ and $\epsilon,\mu > 0$, there exist $c=c(d,\epsilon,\mu)>0$, such that
if $X$ is a $d$-dimensional $(\epsilon,\mu)$-cosystolic expander, then $X$ has the $c$-topological overlapping property.
\end{thm}

Hence, by Gromov's topological overlapping criterion, in order to prove Theorem \ref{thm-main-1}, 
one needs to prove the existence of an infinite family of bounded degree cosystolic expanders.

%%%%%%%%%%%%%%%%%%%%%%%%%%%%%%%%%%%%%%%%%%%%%%%%%%
\subsection{A criterion for cosystolic expansion}
The first infinite family of bounded degree cosystolic expanders was constructed in \cite{KKL},
where it was shown that the $2$-dimensional skeletons (see below) of $3$-dimensional Ramanujan complexes of sufficiently large degree are cosystolic expanders. 

In this work our intention was to generalize the results of \cite{KKL} to all dimensions.
However, the proof method of \cite{KKL} is specifically designed for the two-dimensional case.
Indeed, in order to apply their proof to dimension three or higher, one will need to assume that the graph, whose vertex set is the set of edges in the complex and its edge set is the set of pairs of edges which form a triangle in the complex, is an excellent expander graph. 
As it turns out, this condition is too strong and does not hold in any Ramanujan complex.

The main novelty of our work is to present a new method for transforming expansion in lower dimension to high dimension in the complex.
This allow us to prove our main result, which is a local to global criterion for cosystolic expansion, in all dimensions.
To state this criterion, we need to define the notions of {\em skeletons}, {\em links} and {\em skeleton expander} of a complex. 
\begin{dfn}
Let $X$ be a $d$-dimensional simplicial complex.
\begin{itemize}
\item For $k \leq d$, the $k$-dimensional {\em skeleton} of $X$, denoted $X^{(k)}$, is the complex obtained by deleting from $X$ all faces of dimension greater then $k$.
\item For $\sigma \in X$, the {\em link} of the face $\sigma$ in $X$, denoted $X_\sigma$, is the complex obtained by picking only the faces in $X$ that contains $\sigma$, and removing $\sigma$ from each of these faces. A link is called a proper link of $X$, if $1\leq |\sigma|\leq d-1$.
\item A complex $X$, is said to be an $\alpha$-skeleton expander, for some $\alpha>0$, if its $1$-skeleton satisfy the following graph expansion property
\footnote{The constant $4$ in our definition of skeleton expansion is unnatural, as it implies a notion strictly weaker then graph expansion. See \cite{KM1} for a more accurate notion of skeleton expansion. However, our notion of skeleton expansion will be sufficient for our use here.}:
\begin{equation}
\forall\, A \subset X(0), \quad  \|E(A,A)\| \leq 4 \cdot (\|A\|^2 + \alpha \cdot \|A\|),
\end{equation}
where $E(A,A) \subset X(1)$ are the edges in $X$ with both vertices in $A$.
\end{itemize}
\end{dfn}

We are now ready to state our local to global criterion for cosystolic expansion.
\begin{thm}[Main Theorem] \label{thm-main-2}
For any $d,Q\in \mathbb{N}$ and $\beta > 0$, 
there exists $\epsilon = \epsilon(d,\beta,Q)>0$, $ \mu = \mu(d,\beta) > 0$ and $\alpha = \alpha(d,\beta) > 0$, 
such that if $X$ is a $d$-dimensional complex satisfying:
\begin{itemize}
\item $X$ is $Q$-bounded degree, i.e. $\max_{v\in X(0)}|X_v| \leq Q$.
\item All the proper links of $X$ are $\beta$-coboundry expanders.
\item All the proper links of $X$ and $X$ itself are $\alpha$-skeleton expanders.
\end{itemize}
Then the $(d-1)$-skeleton of $X$ is an $(\epsilon,\mu)$-cosystolic expander.
\end{thm}

It seems very natural to compare our local to global expansion result, with that of Garland \cite{Gar}. 
However, other then philosophically, the two results are unrelated, Garland's work is over $\mathbb{R}$, while ours is over $\mathbb{F}_2$.
Moreover, the proof method of both results is quite different.

%%%%%%%%%%%%%%%%%%%%%%%%%%%%%%%%%%%%%%%%%%%%%%%%%%
\subsection{Ramanujan complexes}
The well known Ramanujan complexes (see \cite{Lub2}), constructed over a fixed residue field, are bounded degree high dimensional complexes, which are excellent skeleton expanders (this follows directly from the Ramanujan property).
Their proper links are spherical buildings admitting a strongly transitive action, and it follows from the work of Gromov \cite{Gro} (see also \cite{LMM}) that they are coboundary expanders.
Therefore, we are left to prove that their links, i.e., that the spherical buildings,
are also good skeleton expanders \footnote
{After a completion of our paper, Izhar Oppenheim has pointed to us that a similar result to Theorem \ref{thm-main-3},
could potentially be deduced from the work of~\cite{Opp}, as a special case of Theorem~8.12 there.
Our proof that exploits the geometric structure of the spherical buildings, allows us to get a proof which is significantly shorter.}.

\begin{thm} \label{thm-main-3}
For any $d\in \mathbb{N}$ and $\alpha>0$, there exist $q_0=q_0(d,\alpha)>0$, such that for any prime power $q\geq q_0$, 
all $d$-dimensional, $q$-thick spherical buildings admitting a strongly transitive action are $\alpha$-skeleton expanders. 
\end{thm}

In conclusion, Ramanujan complexes of sufficiently large but fixed degree satisfies the requirements of Theorem \ref{thm-main-2}, which together with Gromov's topological overlapping criterion (Theorem \ref{thm-Gromov}), implies the following corollary regarding the expansion of Ramanujan complexes.  

\begin{cor} \label{cor-main-4}
For any $d\in \mathbb{N}$, there exists $q_0=q_0(d)$, such that for any prime power $q\geq q_0$, there exists $\mu=\mu(d)>0$, $\epsilon=\epsilon(d,q)>0$ and $c=c(d,q)>0$, such that if $X$ is the $d$-dimensional skeleton of a $(d+1)$-dimensional $q$-thick Ramanujan complex, then
\begin{itemize}
\item $X$ is an $(\epsilon,\mu)$-cosystolic expander.
\item $X$ has the $c$-topological overlapping property.
\end{itemize}
\end{cor}

Finally, by combining the explicit construction of Ramanujan complexes of \cite{LSV2}, with Corollary \ref{cor-main-4}, we get an explicit construction of bounded degree cosystolic expanders with the topological overlapping property in every dimension (Theorem \ref{thm-main-1}).
It is interesting to note that no such random construction, i.e., of bounded degree high dimensional complexes with the topological overlapping property, is currently known.

%%%%%%%%%%%%%%%%%%%%%%%%%%%%%%%%%%%%%%%%%%%%%%%%%%
\subsection{Organization of the paper}
In section \ref{sec-pre} we review some basic definition related to simplicial complexes, norms, links, skeletons and high dimensional expansions.
In section \ref{sec-main} we prove the local to global criterion for cosystolic expansion (Theorem \ref{thm-main-1}).
In section \ref{sec-mix} we prove a one sided mixing lemma for the $1$-skeletons of regular complexes.
In section \ref{sec-build} we review the definition of spherical buildings and show that they are good skeleton expanders. 
Finally, in section \ref{sec-ram}, we prove that Ramanujan complexes satisfies the conditions of Theorem \ref{thm-main-1}, hence their skeletons are therefore examples of bounded degree complexes with the topological overlap property.

%%%%%%%%%%%%%%%%%%%%%%%%%%%%%%%%%%%%%%%%%%%%%%%%%%
\subsection{Acknowledgments}

The authors are grateful to David Kazhdan and Alex Lubotzky for useful discussions and advices, and for Ron Rosenthal for many valuable improvements for this paper.
The first author would like to especially thank Alex Lubotzky for introducing, teaching and encouraging him to work on this problem, without which this work would not have been possible.
We also thank the ERC, and BSF for their support.
This work is part of the Ph.D. thesis of the first author at the Hebrew University of Jerusalem, Israel.

%%%%%%%%%%%%%%%%%%%%%%%%%%%%%%%%%%%%%%%%%%%%%%%%%%%%%%%%%%%%%%%%%%%%%%%%%%%%%%%%%%%%%%%%%%%%%%%%%%%
%%%%%%%%%%%%%%%%%%%%%%%%%%%%%%%%%%%%%%%%%%%%%%%%%%%%%%%%%%%%%%%%%%%%%%%%%%%%%%%%%%%%%%%%%%%%%%%%%%%
\section{Preliminaries on complexes and expansions} \label{sec-pre}
In this section we present the basic definitions and properties of simplicial complexes with norms,
as well as notions of high-dimensional expansions.

%%%%%%%%%%%%%%%%%%%%%%%%%%%%%%%%%%%%%%%%%%%%%%%%%
\subsection{Complexes}
Throughout this paper we shall use the following notations regarding simplicial complexes:

A simplicial complex, $X$, over a set of vertices, $V$, is family of subsets of $V$, $X \subset 2^V$, which is closed under inclusions,
i.e. if $F \in X$ and $E \subset F$ then $E \in X$ (note that the empty-set is always a face in any complex). 
Call the elements of $X$, faces or simplices.
The dimension of a simplex $F \in X$, is defined as $\dim(F) = |F|-1$,
and the dimension of the entire complex is defined as the maximal dimension of a simplex in it, 
i.e. $\dim(X) = \max_{F \in X}\dim(F)$.
A complex is said to be pure if all its maximal faces are of the same dimension.

For convenience sake, by a $k$-face we mean a $k$-dimensional face of the complex, and by a $d$-complex we will always mean a finite $d$-dimensional pure simplicial complex.

Let $X$ be a $d$-complex.
For any $-1 \leq k \leq d$, denote by $X(k)$ the collection of $k$-faces in $X$,
and by $C^k=C^k(X;\mathbb{F}_2) = \{f:X(k) \rightarrow \mathbb{F}_2\}$ the space of $k$-cochains.
As usual, denote by
\begin{equation}
\delta = \delta^k \colon C^k \rightarrow C^{k+1}, \qquad (\delta f)(\tau)= \sum_{\sigma\subset \tau} f(\sigma) 
\end{equation}
the $k$-coboundry map, and by $B^k = B^k(X,\mathbb{F}_2)= Im(\delta^{k-1})$ and $Z^k = Z^k(X,\mathbb{F}_2)= \ker(\delta^k)$
the spaces of $k$-coboundries and $k$-cocycles, respectively.
Recall that $\delta^k \circ \delta^{k-1} = 0$, hence $B^k \subset Z^k$.

Since we are working over $\mathbb{F}_2$, any $k$-cochain is an indicator function of a subset of $k$-faces, and vice versa, any subset of $k$-faces defines a $k$-cochain, i.e. $\{f:X(k) \rightarrow \mathbb{F}_2\} \cong \{A\subset X(k)\}$, where $f\mapsto supp(f)$ and $A\mapsto 1_A$. So, from now we will identify each subset of $k$-faces, $A \subset X(k)$, with the following $k$-cochain $f=1_A \in C^k$.

\begin{dfn}[Norm]
Define the following norm on the space of cochains:
\begin{equation} 
\|\cdot\|=\|\cdot\|^k: C^k \rightarrow [0,1],\quad \|A \| = \sum_{\sigma \in A} w(\sigma)
\quad \mbox{ where } \quad w(\sigma) = \frac{|\{F \in X(d)\,|\, \sigma \subset F\}|}{{d+1 \choose |\sigma|} \cdot |X(d)|}
\end{equation}
\end{dfn}

Note that for any $-1 \leq k \leq d$, and any $A,B \in C^k$ (i.e. $A,B \subset X(k)$), then:
\begin{itemize}
\item $\|A\| = 0$ if and only if $A=\emptyset$, and $\|A\|=1$ if and only if $A=X(k)$.
\item $\|A \cup B\| \leq \|A\| + \|B\|$ with equality if and only if $A$ and $B$ are disjoint.
\end{itemize}

\begin{dfn}[Container] \label{def-container}
Let $X$ be a $d$-complex and let $-1 \leq k \leq r \leq d$.
For any $A \in C^k$, define $\Gamma^r(A) \in C^r$ to be the following $r$-cochain,
\begin{equation}
\Gamma^r(A) = \{F \in X(r)\,|\, \exists \, \sigma \in A,\; \sigma \subset F \}.
\end{equation}
\end{dfn}

\begin{lem} \label{lem-container}
Let $X$ be a $d$-complex and let $-1 \leq k \leq r \leq d$. Then for any $A \in C^k$,
\begin{equation}
\|A\| \leq \|\Gamma^r(A)\| \leq {r+1 \choose k+1} \cdot \|A\|.
\end{equation}
\end{lem}

\begin{proof}
First note that for any $\sigma \in X(k)$, then
\begin{multline}
\|\Gamma^r(\{\sigma\})\| 
= \sum_{\sigma \subset \tau \in X(r)} w(\tau) 
= \frac{1}{|X(d)|} \sum_{\substack{\tau \in X(r) \\ \sigma \subset \tau}} \sum_{\substack{F \in X(d) \\ \tau \subset F }} \frac{1}{{d+1 \choose r+1}}  
= \frac{1}{|X(d)|} \sum_{\substack{F \in X(d) \\ \sigma \subset F }} \sum_{\substack{\tau \in X(r) \\ \sigma \subset \tau \subset F }} \frac{1}{{d+1 \choose r+1}} \\
= \frac{1}{|X(d)|} \sum_{\substack{F \in X(d) \\ \sigma \subset F }} \frac{{d-k \choose r-k}}{{d+1 \choose r+1}} 
= \frac{{d+1 \choose k+1} \cdot {d-k \choose r-k}}{{d+1 \choose r+1}}  w(\sigma) 
= {r+1 \choose k+1} \cdot w(\sigma) 
\end{multline}
So, for any $A\in C^k$, on the one hand,
\begin{equation}
\|\Gamma^r(A)\| = \|\bigcup_{\sigma \in A }\Gamma^r(\{\sigma\})\| \leq \sum_{\sigma \in A} \|\Gamma^r(\{\sigma\})\| = \sum_{\sigma \in A} {r+1 \choose k+1} \cdot w(\sigma) = {r+1 \choose k+1} \cdot \|A\|.
\end{equation}
On the other hand, since each $\tau \in \Gamma^r(A)$ contains at most ${r+1 \choose k+1}$ $k$-faces, i.e. $\tau$ can belong to at most ${r+1 \choose k+1}$ of the $\Gamma^r(\{\sigma\})$, hence, from the previous calculation we get,
\begin{equation}
{r+1 \choose k+1} \cdot \|A\| = \sum_{\sigma \in A} \|\Gamma^r(\{\sigma\})\| 
\leq {r+1 \choose k+1} \cdot \|\bigcup_{\sigma \in A }\Gamma^r(\{\sigma\})\| = {r+1 \choose k+1} \cdot \|\Gamma^r(A)\|,
\end{equation}
which finishes the proof.
\end{proof}

%%%%%%%%%%%%%%%%%%%%%%%%%%%%%%%%%%%%%%%%%%%%%%%%%
\subsection{Links}
Next we introduce the notion of a link of complex, which is a combinatorial analogue of a unit sphere in the complex.
Thus, the links will serve us as local views of the complex. 

\begin{dfn}[Link]
Let $X$ be a $d$-complex and let $\sigma \in X$.
The link of $\sigma$ in $X$ is defined as the following $(d-|\sigma|)$-complex,
\begin{equation}
X_\sigma= \{\tau \in X|\sigma \sqcup \tau \in X\}.
\end{equation}
For any $-1 \leq k \leq d - |\sigma|$, denote by $C^k_\sigma = C^k(X_\sigma,\mathbb{F}_2)$ the space of $k$-cochains of $X_\sigma$,
by $\|\cdot\|_\sigma=\|\cdot\|^k_\sigma: C^k_\sigma \rightarrow [0,1]$ the norm associated to the complex $X_\sigma$,
and by $\delta_\sigma =\delta^k_\sigma : C^k_\sigma \rightarrow C^{k+1}_\sigma$ the $k$-coboundry map.
\end{dfn}

Let us observe the following cases of degenerate links of a $d$-complex:
\begin{itemize}
\item The link of the unique $(-1)$-face, the empty set $\emptyset \in X(-1)$, is the entire complex, $X_\emptyset=X$.
\item The link of a $(d-1)$-face is a $0$-complex, i.e. a collection of isolated vertices.
\end{itemize}
So, in order to avoid these degenerate cases, by a proper link we mean a link of a face $\sigma$ of dimension $0\leq \dim(\sigma)\leq d-2$.

\begin{dfn}[Localization and lifting]
Let $X$ be a $d$-complex and let $\sigma \in X$.
Define the following maps between the original complex, $X$, and the link complex, $X_\sigma$:
\begin{itemize}
\item The first map, called the localization (w.r.t. $\sigma$), takes a cochain of the original complex, and restrict only to the faces which contains $\sigma$, and then delete $\sigma$ from each of them, producing a cochain of the link. Concretely,
\begin{equation} 
I_\sigma : C^* \rightarrow C^{*-|\sigma|}_\sigma, \quad I_\sigma(A) := \{\tau \in X_\sigma\;|\; \tau \sqcup \sigma \in A \}.
\end{equation}
\item The second map, called the lifting (w.r.t. $\sigma$), takes a cochain of the link complex, and adds $\sigma$ to each face in it,
producing a cochain of the original complex. Concretely,
\begin{equation}
I^\sigma : C^*_\sigma \rightarrow C^{*+|\sigma|}, \quad I^\sigma(A) := \{\tau \sqcup \sigma \in X \;|\; \tau \in A  \}.
\end{equation}
\end{itemize}
\end{dfn}

Let us mention a few immediate observations about the localization and lifting operators.
Let $X$ be a $d$-complex, $\sigma \in X$, $|\sigma|\leq k\leq d$, $A\in C_\sigma^{k-|\sigma|}$ and $A'\in C^k$. Then the following holds:
\begin{equation} \label{eq-lift-loc}
I_\sigma \circ I^\sigma(A)=A
\end{equation}
\begin{equation} \label{eq-loc-lift}
I^\sigma \circ I_\sigma(A')=\{\tau \in A' \;|\; \sigma \subset \tau\}
\end{equation}
\begin{equation} \label{eq-lift-cob}
\delta \circ I^\sigma(A) = I^\sigma \circ \delta_\sigma(A)
\end{equation}

The following three Lemmas \ref{lem-norm-global-local}, \ref{lem-norm-inequality} and \ref{norm-link}, describes some relations between the localization and lifting operators and the global and local norms of the complex (i.e. the norms of the original complex $X$ and its links $X_\sigma$, for any $\sigma \in X$).

\begin{lem} \label{lem-norm-global-local}
Let $X$ be a $d$-complex and let $\sigma \in X$.
For any $0 \leq k \leq d-|\sigma|$, and any $A \in C{k-|\sigma|}_\sigma$, then,
\begin{equation}
\|I^\sigma(A)\| = {|\sigma| + k+1 \choose k+1} \cdot w(\sigma) \cdot \|A\|_\sigma.
\end{equation}
\end{lem}

\begin{proof}
Denote by, $w= w_X$, the weight function of the original complex, and by, $w_\sigma = w_{X_\sigma}$, the weight function of the link.
In the language of links, the weight norm is interpreted as $w(\tau) = \frac{|X_\tau(d-|\tau|)|}{{d+1 \choose |\tau|} \cdot |X(d)|}$ (and similarly for $w_\sigma$).
Since the norm of a cochain is define by extending linearly the weight function, it is suffice to show the claim for $A$ which is a singleton.
So, for any $\tau \in X_\sigma(k)$ and $A = \{\tau\} \in C^k_\sigma$, we have
\begin{multline}
w(\sigma) \cdot \|A\|_\sigma =  w(\sigma) \cdot w_\sigma(\tau)
= \frac{|X_\sigma(d-|\sigma|)|}{{d+1 \choose |\sigma|} \cdot |X(d)|}
\cdot \frac{|(X_\sigma)_\tau(d-|\sigma|-|\tau|)|}{{d-|\sigma|+1 \choose |\tau|} \cdot |X_\sigma(d-|\sigma|)|} \\
= \frac{|X_\sigma(d-|\sigma|)|}{{d+1 \choose |\sigma|} \cdot |X(d)|}
\cdot \frac{|X_{\sigma\cup\tau}(d-|\sigma\cup\tau|)|}{{d-|\sigma|+1 \choose |\tau|} \cdot |X_\sigma(d-|\sigma|)|}
= \frac{|X_{\sigma\cup\tau}(d-|\sigma\cup\tau|)|}{{d+1 \choose |\sigma|} \cdot {d-|\sigma|+1 \choose |\tau|} \cdot |X(d)|} \\
= \frac{{d+1 \choose |\sigma \cup \tau|}}{{d+1 \choose |\sigma|} \cdot {d-|\sigma|+1 \choose |\tau|} } \cdot w(\sigma \cup \tau) = \frac{1}{{|\sigma| + |\tau| \choose |\tau|}} \cdot \|I^\sigma(A)\|,
\end{multline}
which finishes the proof.
\end{proof}

\begin{lem} \label{lem-norm-inequality}
Let $X$ be a $d$-complex,$\sigma \in X$ and let $|\sigma| \leq k \leq d$. 
\begin{enumerate}
\item For any $A,B\in C_\sigma^{k-|\sigma|}$, then $\|A\|_\sigma\leq \|B\|_\sigma $ if and only if $  \|I^\sigma (A)\| \leq \|I^\sigma(B)\|$.
\item If $A\in C^k$, $B \in C_\sigma^{k-|\sigma|}$ and $\|A\|\leq \|A+I^\sigma(B)\|$, then $\|I_\sigma (A)\|_\sigma \leq \|I_\sigma(A) + B\|_\sigma$.
\end{enumerate}
\end{lem}

\begin{proof}
\begin{enumerate}
\item Follows immediately from Lemma \ref{lem-norm-global-local}.
\item First note that, every face of $I^\sigma(B)$ contains $\sigma$, i.e. for any $\sigma \not \subset \tau$, $\tau \in A \; \Leftrightarrow \;\tau \in A+I^\sigma(B)$.
Combining this observation with equation \eqref{eq-loc-lift}, we get 
\begin{equation}
\|A\| - \|I^\sigma I_\sigma(A)\|  = 
\sum_{\substack{\tau \in A \\ \sigma \not \subset \tau}} w(\tau) = \sum_{\substack{\tau \in A+I^\sigma(B) \\ \sigma \not \subset \tau}} w(\tau)
= \|A+I^\sigma(B)\| - \|I^\sigma I_\sigma(A+I^\sigma(B))\|.
\end{equation}
Using now the assumption that $\|A\|\leq \|A+I^\sigma(B)\|$, imply that 
\begin{equation}
\|I^\sigma I_\sigma(A)\| \leq \|I^\sigma I_\sigma(A+I^\sigma(B))\|.
\end{equation}
Finally applying fact 1. on the last inequality, together with equation \eqref{eq-lift-loc}, yields
\begin{equation}
\|I_\sigma (A)\|_\sigma \leq \|I_\sigma(A+I^\sigma(B))\|_\sigma = \|I_\sigma(A) + B\|_\sigma
\end{equation}
which finishes the proof.
\end{enumerate}
\end{proof}

\begin{lem} \label{norm-link}
Let $X$ be a $d$-complex, $0 \leq j \leq k \leq d$, and let $A \in C^k$. Then
\begin{equation}
{k+1 \choose j+1} \cdot \|A\| = \sum_{\sigma \in X(j)} \|I^\sigma(I_\sigma(A))\|.
\end{equation}
\end{lem}

\begin{proof}
By the definition of the norm and Fubini's Theorem,
\begin{equation}
\sum_{\sigma \in X(j)} \|I^\sigma(I_\sigma(A))\| = \sum_{\sigma\in X(j)} \sum_{\sigma \subset \tau \in A} w(\tau)
= \sum_{\tau \in A} \sum_{\sigma \subset \tau} w(\tau) = \sum_{\tau \in A} {k+1 \choose j+1} \cdot w(\tau) = {k+1 \choose j+1} \cdot \|A\|.
\end{equation}
\end{proof}

%%%%%%%%%%%%%%%%%%%%%%%%%%%%%%%%%%%%%%%%%%%%%%%%%
\subsection{Skeletons}
Here we define the notions of skeletons of a complex, which are basically what we get when we forget about the higher dimensional faces of the given complex.

\begin{dfn}[Skeleton]
Let $X$ be a $d$-complex and let $0\leq k \leq d$.
Then the $k$-skeleton of $X$ is defined as the following $k$-complex, 
$$X^{(k)} = \{ \tau \in X | \dim(\tau)\leq k\}.$$
\end{dfn}

For example, the $1$-skeleton of a complex is simply its underlying graph.

Note that for any $d$-complex $X$, and any $0\leq t \leq k \leq d$, the spaces of $t$-cochains of the complex and its $k$-skeleton, are the same,
and for $t<k$, the $t$-coboundary operators of the complex and its $k$-skeleton, are also the same.
On the other hand, the norms of the complex and its $k$-skeleton, might be quite different.
However, if the complex is of bounded degree, then the norms of $X$ and $X^{(k)}$ are proportional.

\begin{lem}
Let $t\leq k\leq d,\; Q \in \mathbb{N}$ and let $X$ be a $d$-complex which is $Q$-bounded degree at dimension $k$, i.e. $\max_{\sigma \in X(k)}|X_\sigma(d-|\sigma|)|\leq Q$. 
Then for any $A\in C^t$,
\begin{equation}
\frac{1}{Q \cdot {d-t \choose k-t}} \cdot \|A\|_X 
\leq \|A\|_{X^{(k)}} 
\leq Q\cdot {d+1 \choose k+1}\cdot \|A\|_X
\end{equation} 
\end{lem}

\begin{proof}
Denote by, $w= w_X$, the weight function of the original complex, and by, $w_k = w_{X^{(k)}}$, the weight function of the skeleton.
Since the norm of a cochain is define by extending linearly the weight function, it is suffice to show the claim for $A$ which is a singleton, i.e. it suffice to prove for any $\sigma \in X(t)$,
\begin{equation}
\frac{1}{Q \cdot {d-t \choose k-t}} \cdot w(\sigma) 
\leq w_k(\sigma) \leq Q \cdot {d+1 \choose k+1}\cdot w(\sigma)
\end{equation}
Next note that, since the complex is pure and of $Q$-bounded degree at dimension $k$, then
\begin{equation}
\frac{1}{{d+1 \choose k+1}} \cdot |X(k)| \leq |X(d)| \leq Q\cdot |X(k)|.
\end{equation} 
Similarly, since that for any $\sigma \in X(t)$, the link, $X_\sigma$, is pure and of $Q$-bounded degree at dimension $k-t-1$, then
\begin{equation}
\frac{1}{{d-t \choose k-t}} \cdot |X_\sigma(k-t-1)| \leq |X_\sigma(d-t-1)| \leq Q\cdot |X_\sigma(k-t-1)|.
\end{equation}
we get that for any $\sigma\in X(t)$,
\begin{equation}
w_k(\sigma) 
= \frac{|X_\sigma(k-t-1)|}{{k+1 \choose t+1}\cdot |X(k)|} 
\leq \frac{{d-t \choose k-t}\cdot Q\cdot |X_\sigma(d-t-1)|}{{k+1 \choose t+1}\cdot |X(d)|} 
= {d+1 \choose k+1}\cdot Q \cdot w(\sigma),
\end{equation}
and 
\begin{equation}
w(\sigma) 
= \frac{|X_\sigma(d-t-1)|}{{d+1 \choose t+1}\cdot |X(d)|} 
\leq \frac{Q \cdot {d+1 \choose k+1} \cdot |X_\sigma(k-t-1)|}{{d+1 \choose t+1} \cdot |X(k)|} 
= {d-t \choose k-t} \cdot Q \cdot w_k(\sigma),
\end{equation}
which finishes the proof.
\end{proof}

Next, we define a property of graph expansion for a complex, called skeleton expansion, which  says that the $1$-skeleton of the complex satisfies a form of weak mixing behavior.

\begin{dfn}[Skeleton expander]
Let $X$ be a $d$-complex and let $\alpha >0$. 
Call $X$ an $\alpha$-skeleton expander, if for any $A \subset X(0)$, then
\begin{equation}
\|E(A,A)\| \leq 4 \cdot (\|A\|^2 +   \alpha \cdot \|A\|)
\end{equation}
where $E(A,A) \subset X(1)$ are the edges in $X$ with both vertices in $A$.
\end{dfn}

In section \S \ref{sec-mix}, we prove a one sided mixing Lemma for the skeleton of a regular complex, which gives a criterion for skeleton expansion in terms of non trivial eigenvalues.

%%%%%%%%%%%%%%%%%%%%%%%%%%%%%%%%%%%%%%%%%%%%%%%%%%
\subsection{High dimensional expansion}
Here we present different notions concerning high dimensional expansion for simplicial complexes, which arose in the works of Linial-Meshulam-Wallach \cite{LinM}, \cite{MW} and Gromov \cite{Gro}.

\begin{dfn}[Coboundary and cocycle expansion]
Let $X$ be a $d$-complex and $0 \leq k < d$.
Define the $k$-dimensional cobonudary expansion parameter of $X$ to be:
\begin{equation}
Exp_b^k(X) = \min \{ \frac{\|\delta(A)\|}{\min_{b\in B^k}\|A+b\|} \quad | \quad A \in C^k \setminus B^k  \}.
\end{equation}
Define the $k$-dimensional cocycle expansion parameter of $X$ to be:
\begin{equation}
Exp_z^k(X) = \min \{ \frac{\|\delta(A)\|}{\min_{z\in Z^k}\|A+z\|} \quad | \quad A \in C^k \setminus Z^k  \}.
\end{equation}
\end{dfn}

Let us spell out what both expansion parameters says in the special case of graphs.
In the graph case $d=1$ and $k=0$, the coboundaries are just $V$ and $\emptyset$, and the cocycles are the unions of connected components of the graph.
The cobonudary expansion parameter is equal to the Cheeger constant of the entire graph (up to some normalization), while the cocycle expansion parameter is equal to the minimum of the Cheeger constants in each connected component of the graph.
So, a large cocycle expansion parameter imply that each connected component of the graph is a good expander on its own,
however the graph itself can be disconnected, in particular not an expander.

Let us recall the definition of coboundary expanders, from the introduction.

\begin{dfn}[Coboundary expander]
A $d$-complex, $X$, is said to be an $\epsilon$-coboundry expander, for some $\epsilon >0$,
if $Exp_b^k(X) \geq \epsilon$, for any $0 \leq k \leq d-1$.
\end{dfn}

The notion of coboundary expansion was first originate in the work of \cite{LinM}, in connection to vanishing of cohomology. (The actual term, coboundary expansion, later came from \cite{DK}). 
Recall that for any $0 \leq k \leq d$, each $k$-coboundary is a $k$-cocycle, i.e. $B^k(X;\mathbb{F}_2) \subset Z^k(X;\mathbb{F}_2) $,
and the $k$-th cohomology of $X$ (in $\mathbb{F}_2$-coefficients) is the quotient space
\begin{equation}
H^k(X;\mathbb{F}_2)  = Z^k(X;\mathbb{F}_2) /B^k(X;\mathbb{F}_2).
\end{equation}
It is a simple exercise to see that the following equivalences holds:
\begin{equation}
Exp_b^k(X) > 0 \quad \Leftrightarrow \quad Z^k(X;\mathbb{F}_2) = B^k(X;\mathbb{F}_2)  \quad \Leftrightarrow \quad H^k(X;\mathbb{F}_2) = 0.
\end{equation}
Furthermore, if the $k$-th cohomology is trivial, hence any $k$-cocycle is a $k$-coboundary, then the $k$-coboundary and $k$-cocycle expansion parameters are equivalent:
\begin{equation}
H^k(X;\mathbb{F}_2) = 0 \quad \Rightarrow \quad Exp_b^k(X) = Exp_z^k(X).
\end{equation}
All in all we get the following equivalent characterization for coboundary expansion: 
\begin{equation}
Exp_b^k(X) \geq \epsilon \quad \Leftrightarrow \quad Exp_z^k(X) \geq \epsilon \quad \mbox{and} \quad Z^k = B^k. 
\end{equation}

As noted by Gromov (see also~\cite{DKW}), this notion of vanishing of cohomology is too strong for some application,
since the existence of a cocycle which is not a coboundary, is acceptable just as long as it is not too small.
This is where the definition of cosystoles comes into play.

\begin{dfn}[Cosystoles] \label{def-cosyst}
Let $X$ be a $d$-complex and $0 \leq k \leq d-1$.
Define the $k$-cosystole of $X$ to be the minimal size of a $k$-cococycle which is not a $k$-coboundary, i.e.
\begin{equation}
Syst^k(X) = \min\{\|z\|\;|\;z \in Z^k(X,\mathbb{F}_2) \setminus B^k(X,\mathbb{F}_2)\}.
\end{equation}
\end{dfn}

Finally, we recall the definition of cosystolic expanders from the introduction.

\begin{dfn}[Cosystolic expander]
A $d$-complex $X$ is said to be an $(\epsilon,\mu)$-cosystolic expander, for some $\epsilon, \mu >0$,
if $Exp_z^k(X) \geq \epsilon$ and $Syst^k(X) \geq \mu$, for any $0 \leq k \leq d-1$.
\end{dfn}

\begin{rem}
In recent years, the notion of cosystoles (and its close cousin, systoles) over $\mathbb{F}_2$, have found applications in the field of quantum error correcting codes (see \cite{GL} and \cite{EOT}), where good lower bounds on the systoles and cosystoles give rise to quantum codes with good parameters.
\end{rem}

%%%%%%%%%%%%%%%%%%%%%%%%%%%%%%%%%%%%%%%%%%%%%%%%%%
\subsection{Minimal cochains}
Let us now introduce the following notions of minimal and locally minimal cochains, which we will need later on to prove our local to global cosystolic expansion criterion.

\begin{dfn}[Mininmal and locally minimal]
A cochain $A \in C^k$ is said to be minimal if
 \begin{equation}
\|A\| = \min \{ \|A + b\| \; | \; b \in B^k(X,\mathbb{F}_2)\}.
\end{equation}
A cochain $A \in C^k$ is said to be locally minimal if for any $\emptyset \ne \sigma \in X$,
the localization of $A$ w.r.t. $\sigma$, $I_\sigma(A)$, is a minimal cochain in the link $X_\sigma$.
\end{dfn}

Note that a $d$-complex, $X$, is an $\epsilon$-coboundary expander if and only if for any $0\leq k\leq d-1$, any minimal $k$-cochain, $A\in C^k$, satisfies $\|\delta(A)\| \geq \epsilon \cdot \|A\|$. 

\begin{lem} 
If $A\in C^k$ is a minimal cochain, then $A$ is a locally minimal cochain.
\end{lem}

\begin{proof}
Let $\emptyset \ne \sigma \in X$ and let $\gamma \in C^{k-1-|\sigma|}_\sigma$. 
From the minimality of $A$, equation \eqref{eq-lift-cob}, and Lemma \ref{lem-norm-inequality} part 1, we get that
\begin{equation}
\|A\| \leq \|A+\delta I^\sigma (\gamma)\| = \|A + I^\sigma \delta_\sigma(\gamma)\| 
\; \Rightarrow \; \|I_\sigma(A)\|_\sigma \leq \|I_\sigma(A) + \delta_\sigma(\gamma)\|_\sigma.
\end{equation} 
This proves that $A$ is locally minimal.
\end{proof}

\begin{lem} \label{lem-min}
If $A$ is a minimal cochain and $A' \subset A$, then $A'$ is also a minimal cochain.
\end{lem}

\begin{proof}
First note that, since $A \setminus A'$ and $A'$ are disjoint, $\|A\|=\|A\setminus A'\| + \|A'\|$.
Next note that, since the sum of two cochains is equal to their symmetric difference, for any cochain $c \in C^k$,
\begin{multline}
(A+c) \setminus (A' + c) = ((A \setminus c) \cup (c \setminus A)) \setminus ((A' \setminus c) \cup (c \setminus A')) \subset \\
\subset ((A \setminus c) \setminus (A' \setminus c)) \cup ((c \setminus A) \setminus (c \setminus A'))
= ((A \setminus c) \setminus (A' \setminus c)) \subset A \setminus A'
\end{multline}
where in the second to last step, the equality follows from the fact that $A' \subset A$.
Hence,
\begin{equation} \label{eq-min}
\|A+c\| - \|A' + c\| \leq \|(A+c) \setminus (A' + c)\| \leq \|A \setminus A'\| = \|A\| - \|A'\|
\end{equation}
Now, if $c \in B^k$ is a coboundry, then by the minimality of $A$ we get,
\begin{equation}
\|A'\| = \|A'\| -\|A\| + \|A\| \leq \|A'\| -\|A\| + \|A + c\| \leq \|A' + c\|
\end{equation}
where the last inequality is \eqref{eq-min}, which finishes the proof.
\end{proof}

%%%%%%%%%%%%%%%%%%%%%%%%%%%%%%%%%%%%%%%%%%%%%%%%%%%%%%%%%%%%%%%%%%%%%%%%%%%%%%%%%%%%%%%%%%%%%%%%%%%
%%%%%%%%%%%%%%%%%%%%%%%%%%%%%%%%%%%%%%%%%%%%%%%%%%%%%%%%%%%%%%%%%%%%%%%%%%%%%%%%%%%%%%%%%%%%%%%%%%%
\section{Cosystolic expansion criterion} \label{sec-main}
In this section we prove the following local to global cosystolic expansion criterion. The following result is slightly stronger Theorem \ref{thm-main-2} from the introduction. 

\begin{thm} \label{thm-cri}
For any $d\in \mathbb{N}$, $\beta > 0$ and $Q\in \mathbb{N}$, 
there exists $\epsilon = \epsilon(d,\beta,Q)>0$, $ \mu = \mu(d,\beta) > 0$ and $\alpha = \alpha(d,\beta) > 0$,
such that if $X$ is a $d$-complex satisfying:
\begin{itemize}
\item The complex $X$ is $Q$-bounded degree, i.e. $\max_{v\in X(0)}|X_v| \leq Q$.
\item The link $X_\sigma$ is a $\beta$-coboundry expander, for any $\sigma\in X$, $1\leq |\sigma| \leq d-1$. 
\item The link $X_\sigma$ is an $\alpha$-skeleton expander, for any $\sigma\in X$, $0\leq |\sigma| \leq d-1$.
\end{itemize}
Then for any $0\leq k\leq d-2$ and any $0\leq r\leq d-1$,
\begin{equation} 
Exp_z^k(X)\geq \epsilon \qquad \mbox{ and } \qquad Syst^r(X) \geq \mu
\end{equation}
In particular, the $(d-1)$-skeleton of $X$ is an $(\epsilon,\mu)$-cosystolic expander.
\end{thm}

In order to prove Theorem \ref{thm-cri} we follow \cite{KKL} strategy,
who noticed that the following (co)isoperimetric inequality for small cochains imply cosystolic expansion.

\begin{thm} \label{thm-iso}
For any $d\in \mathbb{N}$ and $\beta > 0$, 
there exists $\bar \epsilon = \bar \epsilon(d,\beta), \bar \mu = \bar \mu(d,\beta) > 0$
and $\alpha = \alpha(d,\beta) > 0$, such that if $X$ is a $d$-complex satisfying:
\begin{itemize}
\item The link $X_\sigma$ is a $\beta$-coboundry expander, for any $\sigma\in X$, $1\leq |\sigma| \leq d-1$. 
\item The link $X_\sigma$ is an $\alpha$-skeleton expander, for any $\sigma\in X$, $0\leq |\sigma| \leq d-1$.
\end{itemize}
Then for any $0\leq k\leq d-1$,
\begin{equation} 
A \in C^k \mbox{ is locally minimal and } \|A\| \leq \bar \mu \quad \Longrightarrow \quad \|\delta(A)\| \geq \bar \epsilon \cdot \|A\|
\end{equation}
\end{thm}

\begin{rem} 
Note that Theorem \ref{thm-iso}, unlike Theorem \ref{thm-cri}, does not require any bounded degree assumption.
This simple fact, makes Theorem \ref{thm-iso} useful also for unbounded degree complexes, as was shown in \cite{LLR}, who proved that for any $d\in \mathbb{N}$, 
there exists $d$-dimensional coboundary expanders, which are of bounded degree at dimension $d-1$, i.e., every $(d-1)$-face is contained in a bounded number of $d$-faces. (The work \cite{LLR} is a generalization of \cite{LubM}).
\end{rem}

\begin{rem} \label{rem-cons}
The constants in Theorem \ref{thm-cri}, are
\begin{multline} 
\mu := \left( \frac{1}{3(d+2)2^{d+3}} \cdot \left(\frac{\beta}{(d+2)\cdot 2^{d+2}}\right)^d \right)^{2^{d+1}}, \quad
\epsilon := \min\{\frac{1}{Q},\mu \}, \quad 
\alpha := \mu^{1+\frac{1}{2^{d+1}}}.
\end{multline}
The constants in Theorem \ref{thm-iso}, are
\begin{multline} 
\bar \mu := \left(\frac{1}{3(d+2)2^{d+3}} \cdot \left(\frac{\beta}{(d+2)\cdot 2^{d+2}}\right)^d \right)^{2^{d+1}}, \quad
\bar \epsilon :=  \frac{1}{3} \cdot \left(\frac{\beta}{(d+2)\cdot 2^{d+2}}\right)^d, \quad 
\alpha := \mu^{1+\frac{1}{2^{d+1}}}.
\end{multline}
We note that we made no attempt to optimize the constants in this work.
\end{rem}

%%%%%%%%%%%%%%%%%%%%%%%%%%%%%%%%%%%%%%%%%%%%%%%%%
\subsection{Sketch of the proof of Theorem \ref{thm-iso}}
Before diving into the proof of Theorem \ref{thm-iso}, which require us to define a few new technical notions, we would like to provide a proof sketch. 

Fix a locally minimal cochain, $A\in C^k$.
We define the notion of fat faces with respect to $A$ in the complex, which essentially says that a face is fat if the localization of $A$ with respect to that face, is large.
Consider the coboundaries of the localization of $A$ with respect to the fat faces, and note that these "local" coboundaries are large by the assumption that the links are cobundary expanders.
The main idea of the proof is that each such "local" coboundary, is either an actual coboundary in the original complex, or it contains a fat face of smaller dimension.
This last claim, is the content of Proposition \ref{pro-seep}, where the term $L_\eta(A,i)$ essentially stands for the sub-cochain of elements in $A$ which contains a fat $i$-faces, 
and $\Upsilon_\eta(A)$ stands for the error term. This error term is negligible due to the skeleton expansion assumption (Proposition \ref{pro-spec}).
So, iterating Proposition \ref{pro-seep}, we get that either $\delta(A)$ is large, or that there is a fat $(-1)$-face.
But the unique $(-1)$-face is fat precisely when $A$ is large, which finishes the proof.

%%%%%%%%%%%%%%%%%%%%%%%%%%%%%%%%%%%%%%%%%%%%%%%%%
\subsection{Fat machinery}
In order to prove the (co)isoperimetric inequality (Theorem \ref{thm-iso}), we first construct a "fat-machinery" which allow us to move calculations
from higher to lower dimensions in the complex.

\begin{dfn}[Fat faces] \label{fat}
Fix a cochain $A \in C^k$ and $0< \eta <1$.
Define inductively the $i$-cochain of fat faces, w.r.t. $A$ and $\eta$, $i=0,\ldots,k$, by
\begin{equation} 
S_\eta^k(A)=A \quad \mbox{ and } \quad S_\eta^{i-1}(A) = \{\sigma \in X(i-1) \, | \, \|I_\sigma(S_\eta^i(A))\|_\sigma \geq \eta^{2^{k-i}} \}
\end{equation}
For $i=-1,0,\ldots,k$, call the elements of $S_\eta^i(A)$ fat faces (w.r.t. $A$ and the fatness constant $\eta$).
\end{dfn}

The following Lemma shows that the sizes of the cochains of fat faces is bounded by the size of the original cochain (up to some constant).

\begin{lem} \label{lem-fatsize}
Let $X$ be a $d$-complex, $-1 \leq i \leq k \leq d$, $A \in C^k$ and $0< \eta <1$. Then,
\begin{equation}
\|S_\eta^i(A)\| \leq \eta^{-2^{k-i}} \cdot \|A\|
\end{equation}
\end{lem}

\begin{proof}
For any, $-1 \leq j \leq k-1$, and any fat $j$-face, $\sigma \in S_\eta^j(A)$, 	
by applying Lemma \ref{lem-norm-global-local} on the cochain $I_\sigma(S_\eta^{j+1}(A)) \in C^0_\sigma$, we get
\begin{equation}
w(\sigma) = \frac{\|I^\sigma(I_\sigma(S_\eta^{j+1}(A)))\|}{(j+2) \cdot \|I_\sigma(S_\eta^{j+1}(A))\|_\sigma}
\leq \frac{\eta^{-2^{k-j-1}}}{j+2} \cdot \|I^\sigma(I_\sigma(S_\eta^{j+1}(A)))\|
\end{equation}
Hence, combining this with Lemma \ref{norm-link},
\begin{equation} \label{eq-fat-norm}
\|S_\eta^j(A)\| = \sum_{\sigma \in S_\eta^j(A)} w(\sigma) 
\leq \frac{\eta^{-2^{k-j-1}}}{j+2} \cdot  \sum_{\sigma \in S_\eta^j(A)} \|I^\sigma(I_\sigma(S_\eta^{j+1}(A)))\|
=  \eta^{-2^{k-j-1}} \cdot \|(S_\eta^{j+1}(A))\|
\end{equation}
Hence, by iterating on equation \eqref{eq-fat-norm} for $j=i,\ldots,k-1$, we get,
\begin{multline}
\|S_\eta^i(A)\| \leq \eta^{-2^{k-i-1}} \cdot \|(S_\eta^{i+1}(A))\| \leq \eta^{-(2^{k-i-1} + 2^{k-i-2})} \cdot \|(S_\eta^{i+2}(A))\| \leq ...  \\
... \leq \eta^{-\sum_{j=i}^{k-1} 2^{k-j}} \|S_\eta^k(A)\| = \eta^{-2^{k-i}+1} \cdot \|A\| \leq \eta^{-2^{k-i}} \cdot \|A\|
\end{multline}
which finishes the proof.
\end{proof}

From Lemma \ref{lem-fatsize}, we get the following consequence, which says that for a small cochain the unique $(-1)$-face is a non-fat face.
(This simple fact will serve as the finishing argument in the proof of Theorem \ref{thm-iso}).
\begin{cor} \label{cor-nofat-1}
Let $X$ be a $d$-complex, $k \leq d$, $A \in C^k$ and $0< \eta <1$.
If $\|A\| < \eta^{2^{k+1}} $ then the unique $(-1)$-face, the empty set is not a fat face, i.e. $\emptyset \not \in S_\eta^{-1}(A)$.
\end{cor}

\begin{proof}
Note that the empty set has the following interesting property: "A local view by $\emptyset$ is everything",
i.e. for any $Y \subset X$ then $I_\emptyset(Y) = Y$ and $\|\cdot\|_\emptyset = \|\cdot\|$.
So, from Lemma \ref{lem-fatsize} and the assumption on the size of $A$, we get
\begin{equation}
\| I_\emptyset(S_\eta^0(A))\|_\emptyset = \|S_\eta^0(A)\| \leq \eta^{-2^k} \cdot \|A\| < \eta^{2^k}
\end{equation}
hence, by definition, the empty set is not a fat face.
\end{proof}

Next, we define the cochain of degenerate faces,
which intuitively one should think of as the error-term when one is trying to move from higher dimension to lower dimension.

\begin{dfn}[Degenerate faces]
Fix a cochain $A \in C^k$ and $0< \eta <1$.
A dead-end is a pair of two equal sized fat faces, $(\sigma,\sigma')$, whose intersection is a codimension-$1$ non-fat face, i.e.
\begin{equation} 
|\sigma|=|\sigma'|=|\sigma \cap \sigma'|+1,\; \sigma, \sigma' \in S_\eta^*(A), \; \sigma \cap \sigma' \not\in S_\eta^{*-1}(A)
\end{equation}
A face $p \in X$ is said to be degenerate if it contains a dead-end in it,
and define $\Upsilon_\eta(A) \in C^{k+1}$ to be the cochain of all $(k+1)$-faces which are degenerate.
\end{dfn}

Next, we define the notion of a fat ladder.
\begin{dfn}[Fat ladders] \label{ladder}
Fix a cochain $A \in C^k$ and $0< \eta < 1$.
For any fat $i$-face, $\sigma \in S_\eta^i(A)$, define the $k$-cochain of fat-ladders of $A$, siting on $\sigma$, to be
\begin{equation} 
L_\eta(A,\sigma) = \{t \in A \, | \, \exists \; \sigma = \sigma_i \subset \ldots \subset \sigma_k =t, \; \sigma_j \in S_\eta^j(A), \; \forall \, j=i,\ldots,k \}
\end{equation}
Define the $k$-cochain of $i$-fat-ladders of $A$, by $L_\eta(A,i) = \bigcup_{\sigma \in S_\eta^i(A)} L_\eta(A,\sigma)$.
\end{dfn}

The next lemma essentially says that if inside a $(k+1)$-face you have a fat ladder and a fat $k$-face, then you can either go down to a deeper ladder, or you contain a dead end, which makes the ambient face degenerate.

\begin{lem} \label{lem-ladd}
Let $X$ be a $d$-complex, $-1 \leq i\leq k \leq d$, $A \in C^k$ and $0< \eta <1$.
Let $\sigma \in S_\eta^i(A)$,  $t \in L_\eta(A,\sigma)$, $t' \in A$ and $t,t' \subset p \in X(k+1)$.
Then, either $t' \in L_\eta(A,t' \cap \sigma)$ or $p \in \Upsilon_\eta(A)$.
\end{lem}

\begin{proof}
By definition, there exists $ \sigma = \sigma_i \subsetneq \ldots \subsetneq \sigma_k =t$, where all $\sigma_j$ are fat.
Define $\sigma'_k=t'$ and $\sigma'_{j-1} = t' \cap \sigma_j = \sigma'_j \cap \sigma_j$ for any $j=i,\ldots,k$.
If all the $\sigma'_j$ are fat, and since $t' \cap \sigma = \sigma'_{i-1} \subset \ldots \subset \sigma'_k =t'$,
so by removing repetitions if needed ($\sigma'_j = \sigma'_{j+1}$) we get that $t' \in L_\eta(A,t' \cap \sigma)$.
Otherwise, there is a non-fat $\sigma'_j$, and w.l.o.g. we may assume $j$ is maximal, i.e. $\sigma'_{j+1}$ is fat,
and since that $\sigma_{j+1} \not \subset t'$ (otherwise $\sigma'_j = \sigma_{j+1}$ is fat),
we get that $|\sigma_{j+1}| = |\sigma'_{j+1}| = |\sigma_{j+1}\cap \sigma_{j+1}| -1$, hence $p \in \Upsilon_\eta(A)$.
\end{proof}

%%%%%%%%%%%%%%%%%%%%%%%%%%%%%%%%%%%%%%%%%%%%%%%%%
\subsection{Proof of Theorem \ref{thm-iso}}

The following Proposition is the heart of the "fat machinery" which will allow us to move calculations from higher to lower dimensions.
\begin{pro} \label{pro-seep}
Let $k< d \in \mathbb{N}$, $0< \eta <1$, $0 < \beta$ and let
$X$ be a $d$-complex such that for any $\sigma \in X$, $1\leq |\sigma|\leq d-1$, the link, $X_\sigma$, is a $\beta$-coboundary expander.
Then for any locally minimal $k$-cochain, $A \in C^k$, and any $0\leq i \leq k$,
\begin{equation}
\frac{\beta}{{k+2 \choose i+1}} \cdot \|L_\eta(A,i)\| \leq \|\delta(A)\| + (k+2) \cdot \|L_\eta(A,i-1)\| + \|\Upsilon_\eta(A)\|
\end{equation}
\end{pro}

\begin{proof}
Let us evaluate the following expression,
\begin{equation}
R = \bigcup_{\sigma \in S^i(A)} I^\sigma \circ \delta_\sigma \circ I_\sigma (L_\eta(A,\sigma)) \subset X(k+1).
\end{equation}
On the one hand, since $A$ is locally minimal, then $I_\sigma(A)$ is minimal, and $I_\sigma(L_\eta(A,\sigma)) \subset I_\sigma(A)$, so by Lemma \ref{lem-min} , $I_\sigma(L_\eta(A,\sigma))$ is also minimal.
So, by the assumption that all the proper links of $X$ are $\beta$-coboundry expanders, we get that for any $\sigma \in S_\eta^i(A)\subset X(i)$, 
\begin{equation} \label{eq-311}
 \beta \cdot \|I_\sigma (L_\eta(A,\sigma))\|_\sigma \leq \|\delta_\sigma \circ I_\sigma(L_\eta(A,\sigma))\|_\sigma.
\end{equation} 
Note that $L_\eta(A,\sigma) \subset \{\tau \in X(k)|\sigma \subset \tau\}$, hence by equation \eqref{eq-loc-lift}, $I^\sigma \circ I_\sigma (L_\eta(A,\sigma)) = L_\eta(A,\sigma)$.
Combining this with equation \eqref{eq-311}, Lemma \ref{lem-norm-inequality} part 1, and the fact that each element of $R$ contains at most ${k+2 \choose i+1}$ faces from $S_\eta^i(A)\subset X(i)$, we get
\begin{equation} \label{eq-pro-seep-1}
\beta \cdot \|L_\eta(A,i)\| \leq \sum_{\sigma \in S_\eta^i(A)} \beta \cdot \|L_\eta(A,\sigma)\| \leq \sum_{\sigma \in S_\eta^i(A)} \|I^\sigma \circ \delta_\sigma \circ I_\sigma(L_\eta(A,\sigma))\| \leq {k+2 \choose i+1}\cdot\|R\|.
\end{equation}
On the other hand, for any $\sigma \in S_\eta^i(A)$ and any $p \in I^\sigma \circ \delta_\sigma \circ I_\sigma(L_\eta(A,\sigma))$, one of the following three possibilities must occur:
\begin{enumerate}
\item All the $k$-faces in $p$ which belongs to $A$, contains $\sigma$, and they are all from $L_\eta(A,\sigma)$. 
In this case $p \in \delta(A)$.
\item All the $k$-faces in $p$ which belongs to $A$, contains $\sigma$, but not all of them are from $L_\eta(A,\sigma)$.
In this case there is some $t' \in A \setminus L_\eta(A,\sigma)$ such that $\sigma \subset t' \subset p$.
Now, since $p \in I^\sigma \circ \delta_\sigma \circ I_\sigma(L_\eta(A,\sigma))$, there must be at least one $t \in L_\eta(A,\sigma)$ such that $t \subset p$.
So, by Lemma \ref{lem-ladd} we get that $p \in \Upsilon_\eta(A)$.
\item There is a $k$-face in $p$ which belongs to $A$, and it does not contain $\sigma$.
In this case there is some, $t' \in A$, such that $\sigma \not \subset t' \subset p$.
Like before, since $p \in I^\sigma \circ \delta_\sigma \circ I_\sigma(L_\eta(A,\sigma))$, there must be at least one $t \in L_\eta(A,\sigma)$ such that $t \subset p$.
And again by Lemma \ref{lem-ladd} we get that either $p \in \Upsilon_\eta(A)$, or $t' \in L_\eta(A,t' \cap \sigma) \subset L_\eta(A,i-1)$,
i.e. $p$ is a $(k+1)$-face which contains a $k$-face from $L_\eta(A,i-1)$, hence by definition \ref{def-container} $p \in \Gamma^{k+1}(L_\eta(A,i-1))$.
\end{enumerate}
In conclusion we get that 
\begin{equation} \label{eq-pro-seep-2}
R \subset \delta(A)  \cup \Gamma^{k+1}(L_\eta(A,i-1)) \cup \Upsilon_\eta(A).
\end{equation}
Combining equation \eqref{eq-pro-seep-1} and \eqref{eq-pro-seep-2}, together with Lemma \ref{lem-container}, yields
\begin{equation}
\frac{\beta}{{k+2 \choose i+1}}  \cdot \|L_\eta(A,i)\| \leq \|R\| \leq \|\delta(A)\| + (k+2) \cdot \|L_\eta(A,i-1)\|  + \|\Upsilon_\eta(A)\|
\end{equation}
which finishes the proof.
\end{proof}

The following Proposition gives an effective bound on the cochain of degenerate faces in terms of the skeleton expansion and the fatness constant.
\begin{pro} \label{pro-spec}
Let $k< d \in \mathbb{N}$, $0< \eta <1$, $0 < \alpha \leq \eta^{2^{k+1}}$ and let
$X$ be a $d$-complex such that for any $\sigma \in X$, $0\leq |\sigma|\leq d-1$, the link, $X_\sigma$, is an $\alpha$-skeleton expander.
Then for any $k$-cochain $A \in C^k$, 
\begin{equation}
\|\Upsilon_\eta(A) \| \leq  (k+2)\cdot 2^{k+4} \cdot \eta \cdot \|A\|.
\end{equation}
\end{pro}

\begin{proof}
For any $t \leq d$ and any $t$-cochain $Y$, denote by $E(Y,Y)$ the $(t+1)$-cochain of $(t+1)$-faces which contains at least two different $t$-faces from $Y$,
and let $\Gamma^{k+1}(Y) \in C^{k+1}$ be as in Lemma \ref{lem-container}.
Then, by the definition of the fat-degenerate faces, we get
\begin{equation}
\Upsilon_\eta(A) 
= \bigcup_{i=-1}^{k-1} \bigcup_{\sigma \in X(i) \setminus S_\eta^i(A)} \Gamma^{k+1}(I^\sigma( E(I_\sigma(S_\eta^{i+1}(A)),I_\sigma(S_\eta^{i+1}(A)))))
\end{equation}
So, by Lemma \ref{lem-container}, we get
\begin{equation} \label{eq-mix-cor1}
\|\Upsilon_\eta(A)\| 
\leq \sum_{i=-1}^{k-1} \sum_{\sigma \in X(i) \setminus S_\eta^i(A)} {k+2 \choose i+2} \cdot \|I^\sigma(E(I_\sigma(S_\eta^{i+1}(A)),I_\sigma(S_\eta^{i+1}(A))))\|
\end{equation}
From the skeleton expansion, we get for any $\sigma \in X(i)$,
\begin{equation}
\|E(I_\sigma(S_\eta^{i+1}(A)),I_\sigma(S_\eta^{i+1}(A)))\|_\sigma 
\leq 4\cdot \|I_\sigma(S_\eta^{i+1}(A))\|_\sigma \cdot ( \|I_\sigma(S_\eta^{i+1}(A))\|_\sigma + \alpha)
\end{equation}
Now,  by the Lemma \ref{lem-norm-global-local}, we can multiply both sides by ${k \choose |\sigma|} \cdot w(\sigma)$, and get
\begin{equation}
\|I^\sigma(E(I_\sigma(S_\eta^{i+1}(A)),I_\sigma(S_\eta^{i+1}(A))))\| 
\leq 4 \cdot \|I^\sigma(I_\sigma(S_\eta^{i+1}(A)))\| \cdot ( \|I_\sigma(S_\eta^{i+1}(A))\|_\sigma + \alpha)
\end{equation}
Next, if $\sigma \in X(i) \setminus S_\eta^i(A)$, then by the definition of fat faces,
\begin{equation}
\|I^\sigma(E(I_\sigma(S_\eta^{i+1}(A)),I_\sigma(S_\eta^{i+1}(A))))\| 
\leq 4 \cdot \|I^\sigma(I_\sigma(S_\eta^{i+1}(A)))\| \cdot (\eta^{2^{k-i}} + \alpha)
\end{equation}
Summing this over all non-fat $i$-faces,
\begin{multline}
\sum_{\sigma \in X(i) \setminus S_\eta^i(A)} \|I^\sigma(E(I_\sigma(S_\eta^{i+1}(A)),I_\sigma(S_\eta^{i+1}(A))))\|
\leq  4(\eta^{2^{k-i}} + \alpha) \cdot \sum_{\sigma \in X(i) \setminus S_\eta^i(A)}\|I^\sigma(I_\sigma(S_\eta^{i+1}(A)))\| \\
\leq 4(\eta^{2^{k-i}} + \alpha) \cdot \sum_{\sigma \in X(i)} \|I^\sigma(I_\sigma(S_\eta^{i+1}(A)))\| 
= 4(\eta^{2^{k-i}} + \alpha) \cdot (i+2) \cdot \|S_\eta^{i+1}(A)\|
\end{multline}
where the last equality follows from Lemma \ref{norm-link}.
Applying Lemma \ref{lem-fatsize}, we get
\begin{equation} \label{eq-mix-cor2}
\sum_{\sigma \in X(i) \setminus S_\eta^i(A)} \|I^\sigma(E(I_\sigma(S_\eta^{i+1}(A)),I_\sigma(S_\eta^{i+1}(A))))\| 
\leq 4(i+2) \cdot (\eta^{2^{k-(i+1)}} + \alpha \cdot \eta^{-2^{k-(i+1)}}) \cdot \|A\|
\end{equation}
Combining equations \eqref{eq-mix-cor1} and \eqref{eq-mix-cor2} together, we get
\begin{multline}
\|\Upsilon_\eta(A)\| \leq \sum_{i=-1}^{k-1} {k+2 \choose i+2} \cdot 4(i+2) \cdot (\eta^{2^{k-(i+1)}} + \alpha \cdot \eta^{-2^{k-(i+1)}}) \cdot \|A\| \\
\leq \left(4 \cdot \sum_{i=-1}^{k-1} {k+2 \choose i+2} \cdot (i+2)\right) \cdot (\eta^{2^{k-k}} + \alpha \cdot \eta^{-2^{k-0}})  \cdot \|A\|  
\leq \left(4(k+2) \cdot \sum_{i=-1}^{k-1} {k+1 \choose i+1}\right) \cdot (\eta + \alpha \cdot \eta^{-2^k})  \cdot \|A\| \\
\leq (k+2)\cdot 2^{k+3} \cdot (\eta + \alpha \cdot \eta^{-2^k})  \cdot \|A\|
\leq (k+2)\cdot 2^{k+3} \cdot 2\eta \cdot \|A\|
\end{multline}
which finishes the proof.
\end{proof}

Finally, we are able to prove the isoperimetric inequality for small cochains (Theorem \ref{thm-iso}).

\begin{proof}[Proof of Theorem \ref{thm-iso}]
Define the following constants,
\begin{multline} 
c_0 := \frac{\beta}{(k+2)\cdot 2^{k+2}}, \quad
\bar \epsilon := \frac{ c_0^k}{2} , \quad 
\eta := \frac{\bar \epsilon}{(k+2)2^{k+4}} , \quad
\bar \mu := \eta^{2^{k+1}} \quad 
\mbox{ and } \quad \alpha := \eta^{2^{k+1}+1}.
\end{multline}
Note that by definition $L_\eta(A,k)=A$, and if $ \|A\| \leq \bar \mu = \eta^{2^{k+1}}$,
then by Corollary \ref{cor-nofat-1} the only $(-1)$-face, the empty set, is non fat (w.r.t. $A$ and $\eta$), and hence $\|L_\eta(A,-1)\|= 0$.
Therefore, for any constant $c \leq 1$ we get,
\begin{equation} \label{eq-telescope}
c^k \cdot \|A\| = \sum_{i=0}^k c^{i-1} (c \cdot \|L_\eta(A,i)\| -  \|L_\eta(A,i-1)\|) \leq \sum_{i=0}^k (c \cdot \|L_\eta(A,i)\| -  \|L_\eta(A,i-1)\|).
\end{equation}
Note that $c_0 \leq 1$ and that $(k+2)\cdot c_0 \leq \frac{\beta}{{k+2 \choose i+1}}$, for any $0 \leq i \leq k$.
So, by applying Proposition \ref{pro-seep} on equation \eqref{eq-telescope}, with the constant $c_0$, we get
\begin{equation}
c_0^k \cdot \|A\| \leq \frac{1}{k+2}\cdot\sum_{i=0}^k (\frac{\beta}{{k+2 \choose i+1}} \cdot \|L_\eta(A,i)\| - (k+2)\cdot \|L_\eta(A,i-1)\|)
\leq \|\delta(A)\| + \|\Upsilon_\eta(A)\|.
\end{equation}
Combining this with Proposition \ref{pro-spec}, we get
\begin{equation} \label{eq-end}
\|\delta(A)\| \geq (c_0^k - (k+2)2^{k+4}\cdot \eta) \cdot \|A\| = \bar \epsilon \cdot \|A\|,
\end{equation}
which finishes the proof.
\end{proof}

%%%%%%%%%%%%%%%%%%%%%%%%%%%%%%%%%%%%%%%%%%%%%%%%%
\subsection{Proof of Theorem \ref{thm-cri}}

The fact that a coisoperimetric inequality for small cochains (Theorem \ref{thm-iso}), implies a cosystolic expansion (Theorem \ref{thm-cri}),
was first shown in \cite[\S~4]{KKL}, but for the sake of being self-contained we add here their argument.

\begin{pro}\cite[Proposition~2.5]{KKL} \label{KKL-2.5}
Let $X$ be a $d$-complex, and define $Q = \max_{v \in X(0)}|X_v|$. Let $0 \leq k \leq d$.
Then for any $A \in C^k$, there exists $\gamma \in C^{k-1}$, which satisfies:
\begin{enumerate}
\item The cochain $A + \delta(\gamma)$ is locally minimal.
\item $ \|A + \delta(\gamma)\| \leq \|A\|$.
\item $\|\gamma\| \leq Q \cdot \|A\|$.
\end{enumerate}
\end{pro}

\begin{proof}
First note that $N(A) := {d+1 \choose k+1} \cdot |X(d)| \cdot \|A\|$ is a non-negative integer. We prove the claim by induction on $N(A)$.
In the base case $N(A)=0$, then $A = \emptyset \in C^k$ is the empty $k$-cochain, and the claim holds for $\gamma = \emptyset \in C^{k-1}$ the empty $(k-1)$-cochain.
Assume the claim holds for all cochains $A' \in C^k$ such that $N(A') < N(A)$, i.e. such that $\|A'\| < \|A\|$.

If $A$ is locally minimal, then the claim holds for $\gamma = \emptyset \in C^{k-1}$ the empty $(k-1)$-cochain.
Otherwise, there exist $\emptyset \ne \sigma \in X$, and some $c \in C^{k-1-|\sigma|}_\sigma $, such that
\begin{equation}
\|I_\sigma(A) + \delta_\sigma(c)\|_\sigma < \|I_\sigma(A)\|_\sigma 
\end{equation}
Denote $c'=I^\sigma(c)\in C^{k-1}$. Then by equation \eqref{eq-lift-cob} and Lemma \ref{lem-norm-inequality} part 2, we get
\begin{equation}
\|A+\delta(c')\| <  \|A\|.
\end{equation}
So, $N(A+\delta(c')) < N(A)$, and since both are natural numbers, then $N(A+\delta(c')) \leq N(A)-1$. By the induction assumption there exist $\gamma' \in C^{k-1}$, such that:
\begin{enumerate}
\item $(A+\delta(c'))+\delta(\gamma') = A+\delta(c' + \gamma')$ is a locally minimal cochain,
\item $\|A+\delta(c' + \gamma')\| \leq \|A+\delta(c')\| < \|A\| $, and 
\item $\|\gamma'\| \leq Q \cdot \|A+\delta(c')\| \leq Q \cdot (\|A\| - \frac{1}{{d+1 \choose k+1} \cdot |X(d)|})$.
\end{enumerate}
Hence by taking $\gamma = \gamma' + c$, and noting that $\|c\| \leq \frac{Q}{{d+1 \choose k+1} \cdot |X(d)|}$ (since $c \subset X_\sigma$ and $|X_\sigma|\leq Q$), we get that $\|\gamma\| \leq \|\gamma'\|+\|c\| \leq  Q\cdot \|A\|$, which finishes the proof.
\end{proof}

Now, using Proposition \ref{KKL-2.5} and the isoperimetric inequality for small cochains (Theorem \ref{thm-iso}), we are able to prove the cosystolic expansion criterion (Theorem \ref{thm-cri}).
\begin{proof}[Proof of Theorem \ref{thm-cri}]
Let $\bar \mu$ and $\bar \epsilon$ be the constants from Theorem \ref{thm-iso} and let 
$Q=\max_{v \in X(0)} |X_v|$. Define  $\epsilon = \min\{\bar \mu, \frac{1}{Q}\}$ and $\mu = \bar \mu$.

We begin by proving the cocycle expansion.
Let $A \in C^k$. First note that if $\|\delta(A)\| \geq \mu$, and since $\|A\| \leq 1$ for any cochain, then
\begin{equation}
\|\delta(A) \| \geq \mu \geq \epsilon \geq \epsilon \cdot \|A\|
\end{equation}
So, let us assume $\|\delta(A)\| \leq \mu$, and let $\gamma \in C^k$, be as in Proposition \ref{KKL-2.5} apply on the cochain $\delta(A)$.
Then $\delta(A)+\delta(\gamma)=\delta(A+\gamma)$ is a locally minimal cochain and $\|\delta(A+\gamma)\| \leq \|\delta(A)\| \leq \mu$. 
By Theorem \ref{thm-iso} and the fact that $\delta \circ \delta = 0$ we get that
\begin{equation}
0=\|\delta(\delta(\alpha+\gamma))\| \geq \bar \epsilon \cdot \|\delta(\alpha+\gamma)\| \quad \Rightarrow \quad \delta(\alpha+\gamma)=0.
\end{equation}
So, $A+\gamma \in Z^k$, and hence $\gamma =  A + (A + \gamma) \in \{ A +z \; |\; z \in Z^k\}$.
Now, by Proposition \ref{KKL-2.5} part (3), $\|\gamma\| \leq Q \cdot \|\delta(A)\|$, and we get
\begin{equation}
\|\delta(A) \| \geq \frac{1}{Q} \cdot \|\gamma\| \geq \epsilon \cdot \|\gamma\| = \epsilon \cdot \|A + (A+\gamma)\| \geq \epsilon \cdot \min_{z \in Z^k} \|A + z\|
\end{equation}
which gives us the cocycle expansion $Exp_z^k(X) \geq \epsilon$.

Next we prove the cosystolic bound.
Let $A \in Z^k \setminus B^k$ (if no such cocycle exists there is nothing to prove).
By Proposition \ref{KKL-2.5}, let $A' = A + \delta(\gamma)$ be such that $A'$ is locally minimal and $\|A'\| \leq \|A\|$.
Note that since $A \in Z^k \setminus B^k$ and $\delta(\gamma) \in B^k$ then $A'\in Z^k \setminus B^k$ also.
If $\|A'\| \leq \mu = \bar \mu$, then by Theorem \ref{thm-iso} and the fact that $A'$ is a cocycle, we get
\begin{equation}
0 = \|\delta(A')\| \geq \bar \epsilon \cdot \|A'\| \quad \Rightarrow \quad A' = 0
\end{equation}
which is a contradiction since, $0 \in B^k$, and $A'$ is not in $B^k$. So, $\|A\| \geq \|A'\| \geq \mu$, which gives us the cosystolic bound $Syst^k(X) \geq \mu$.
\end{proof}

%%%%%%%%%%%%%%%%%%%%%%%%%%%%%%%%%%%%%%%%%%%%%%%%%%%%%%%%%%%%%%%%%%%%%%%%%%%%%%%%%%%%%%%%%%%%%%%%%%
%%%%%%%%%%%%%%%%%%%%%%%%%%%%%%%%%%%%%%%%%%%%%%%%%%%%%%%%%%%%%%%%%%%%%%%%%%%%%%%%%%%%%%%%%%%%%%%%%%
\section{Skeleton mixing lemma} \label{sec-mix}
The purpose of this section is to prove a one sided mixing lemma for the $1$-skeleton of a regular complex (see below), i.e. giving a spectral criterion for skeleton-expansion.

\begin{rem}
We note that mixing behavior, spectral gaps and random walks of high dimensional complexes, are subjects that have been already intensively studied in several works (see \cite{EGL}, \cite{GS}, \cite{GW}, \cite{KM1}, \cite{KM2}, \cite{KR}, \cite{Opp} \cite{Par}, \cite{PRT}, \cite{PR} and \cite{Ros}). The mixing Lemma that we present here, is much simpler then the ones appearing in the above mentioned works, and it is actually a result about partite-regular graph, rather then complexes.
\end{rem}

\begin{dfn}[Regular complex]
A $d$-complex $X$, is said to be regular if it satisfies:
\begin{itemize}
\item There exist a partition of the vertices $X(0) = \bigsqcup_{i=0}^d V_i$, such that $X(d) \subset \prod_{i=0}^d V_i$.
\item For any $I \subset J \subset [d]:=\{0,1,\ldots,d\}$, there exist $k_I^J \in \mathbb{N}$, such that each $p \in X \cap \prod_{i \in I} V_i$,
is contained in exactly $k_I^J$ faces from $X \cap \prod_{j\in J} V_j$.
\end{itemize}
\end{dfn}

For example, in the simplest case where $X$ is of dimension $1$, i.e. a graph,
then $X$ is a regular complex if and only if $X$ is a bipartite biregular graph.

\begin{rem} \label{rem-regular}
Note that if $X$ is a regular complex, then so does all of its links and skeletons.
\end{rem}

Next, we define what is the "second eigenvalue" of the $1$-skeleton of a regular complex.

\begin{dfn}[Non-trivial eigenvalue] \label{def-typeind-bipartite}
Let $X$ be a regular $d$-complex.
For any $0 \leq i < j \leq d$, define the $(i,j)$-type induced bipartite graph to be $X_{(i,j)} = (V_i \bigsqcup V_j, X(1) \cap V_i \times V_j)$.
Denote by $\lambda(X_{(i,j)})$ its normalized second largest eigenvalue.
Define the normalized largest non-trivial eigenvalue of (the $1$-skeleton of) $X$ to be
$$\lambda(X) := \max_{i \ne j} \lambda(X_{(i,j)}).$$
\end{dfn}

Let us now state the one sided mixing lemma for the $1$-skeleton of a regular complex.

\begin{pro} \label{partite-mix}
Let $X$ be a regular complex, and let $\lambda(X)$ be its normalized largest non-trivial eigenvalue. Then, for any $A,B \subset X(0)$,
\begin{equation}
\|E(A,B)\| \leq 2(\frac{d+1}{d})\cdot (\|A\| \cdot \|B\| +   \lambda(X) \cdot \sqrt{\|A\| \cdot \|B\|})
\end{equation}
where $E(A,B) \subset X(1)$ are the edges in $X$ with vertices from both $A$ and $B$.
\end{pro}

Note that in the $1$-dimensional case, such a Mixing Lemma is already known.

\begin{lem}\cite[Corollary 3.4]{EGL} \label{lem-bimix}
Let $G=(V_1 \bigsqcup V_2, E)$ be a bipartite biregular graph, and let $\lambda(G)$ be its normalized second largest eigenvalue. Then, for any $A \subset V_1$ and any $B\subset V_2$, 
\begin{equation}
\vert \frac{|E(A,B)|}{|E(X)|} - \frac{|A|}{|V_1|}\frac{|B|}{|V_2|} \vert \leq \lambda(G) \cdot \sqrt{\frac{|A|}{|V_1|}\frac{|B|}{|V_2|}}
\end{equation}
\end{lem}

This bipartite mixing lemma will imply the general skeleton mixing lemma.

\begin{proof}[Proof of Proposition \ref{partite-mix}]
First note that since $X$ is a regular complex, and let $X(0) = \bigsqcup_{i=0}^d V_i$ be the partition, then for any $I \subset [d]$ and $A \subset X \cap \prod_{i \in I} V_i$,
\begin{equation} 
{d+1 \choose |I|} \cdot \|A\| = \sum_{\sigma \in A} \frac{|\{F \in X(d)\,|\,\sigma \in F\}|}{|X(d)|}
= \frac{|A| \cdot k_I^{[d]}}{|X \cap \prod_{i \in I} V_i|\cdot k_I^{[d]} } = \frac{|A| }{|X \cap \prod_{i \in I} V_i|}
\end{equation}
In particular, for any $i,j\in [d]$ and any $A \subset V_i, B \subset V_j$,
\begin{equation} 
\|A\| = \frac{1}{{d+1 \choose 1}} \frac{|A|}{|V_i|},\; \|B\| = \frac{1}{{d+1 \choose 1}}\frac{|B|}{|V_j|}, \mbox{ and }
\|E(A,B)\| = \frac{1}{{d+1 \choose 2}} \frac{|E(A,B)|}{|X(1)\cap V_i \times V_j|}
\end{equation}
So restating Lemma \ref{lem-bimix} in terms of the complex norm, we get for any $A \subset V_i, B \subset V_j$,
\begin{multline}
\|E(A,B)\| \leq \frac{{d+1 \choose 1}^2}{{d+1 \choose 2}} \|A\| \cdot \|B\| + \lambda(X_{(i,j)}) \frac{{d+1 \choose 1}}{{d+1 \choose 2}} \sqrt{\|A\| \cdot \|B\|}\\
\leq 2(\frac{d+1}{d})\cdot  (\|A\| \cdot \|B\| + \frac{\lambda(X)}{d+1} \cdot \sqrt{\|A\| \cdot \|B\|})
\end{multline}  	
Now, let $A,B \subset X(0)$, and denote $A_i = A \cap V_i$ and $B_i = B \cap V_i$ for any $0 \leq i \leq d$.
Since $X$ is partite, $E(A,B) = \bigsqcup_{i \ne j} E(A_i,B_j)$, hence
\begin{equation} 
\|E(A,B)\| = \sum_{i \ne j} \|E(A_i,B_j)\| \leq 2(\frac{d+1}{d})\cdot \sum_{i \ne j} (\|A_i\| \cdot \|B_j\| + \frac{\lambda(X)}{d+1}\sqrt{\|A_i\| \cdot \|B_j\|})
\end{equation}
Similarly, since $A = \bigsqcup_i A_i$ and $B = \bigsqcup_j B_j$, then
\begin{equation}
\sum_{i \ne j} \|A_i\| \cdot \|B_j\| \leq (\sum_{i} \|A_i\|  )\cdot(\sum_{j}  \|B_j\| ) = \|A\| \cdot \|B\|
\end{equation}
Next, note that for any $N$ non-negative numbers $x_1,\ldots,x_N \in \mathbb{R}_{\geq 0}$, one has
\begin{equation}
(\sum_{i=1}^N x_i)^2 \leq N \cdot \max_{1 \leq i \leq N} (x_i^2) \leq N \cdot \sum_{i=1}^N x_i^2
\end{equation}
Applying this for $N=(d+1)^2$ and $x_i = \sqrt{\|A_i\|\cdot \|B_j\|}$, we get
\begin{equation}
\sum_{i,j} \sqrt{\|A_i\| \cdot \|B_j\|} \leq \sqrt{(d+1)^2 \cdot \sum_{i, j} \|A_i\| \cdot \|B_j\|} \leq (d+1) \cdot \|A\| \cdot \|B\|
\end{equation}
which finishes the proof.
\end{proof}

In particular, since $2(\frac{d+1}{d}) \leq 4$ for any $d \in \mathbb{N}$, we get the following.
\begin{cor} \label{cor-skeleton-expansion}
Let $X$ be a regular $d$-complex. Then $X$ is an $\lambda(X)$-skeleton expander.
\end{cor}

%%%%%%%%%%%%%%%%%%%%%%%%%%%%%%%%%%%%%%%%%%%%%%%%%%%%%%%%%%%%%%%%%%%%%%%%%%%%%%%%%%%%%%%%%%%%%%%%%%%
%%%%%%%%%%%%%%%%%%%%%%%%%%%%%%%%%%%%%%%%%%%%%%%%%%%%%%%%%%%%%%%%%%%%%%%%%%%%%%%%%%%%%%%%%%%%%%%%%%%
\section{Spherical buildings} \label{sec-build}
The object of this section is to introduce the notion of spherical buildings, and to show that they are good skeleton expanders.

%%%%%%%%%%%%%%%%%%%%%%%%%%%%%%%%%%%%%%%%%%%%%%%%%
\subsection{Definition of spherical buildings}
Here we give a definition of spherical buildings, and list some of their properties which we shall use.
For more on buildings we refer to \cite{AB}.

Before defining a building, let us first define the notion of a chamber complex.

\begin{dfn}[Chamber complex]
A $d$-dimensional simplicial complex, $X$, is said to be a chamber complex,
if it is pure (i.e. all maximal faces are $d$-dimensional), and for any two maximal faces $C,C' \in X(d)$,
there is a sequence of $d$-faces, $C=C_1,\ldots,C_n=C'$, such that for each $i=1,\ldots,n-1$, the intersection $C_i \cap C_{i+1}$
is a $(d-1)$-face.

For a $d$-dimensional chamber complexes, call a $d$-face a chamber, call a $(d-1)$-face a panel,
and call a sequence as above, $C=C_1,\ldots,C_n=C'$, a gallery from $C$ to $C'$.

A chamber complex, $X$, is said to be thin if each panel is contained in exactly $2$ chambers,
and it is said to be $q$-thick, $q > 1$, if each panel is contained in exactly $q+1$ chambers.
By a thick chamber complex, we mean $q$-thick for some $q > 1$.
\end{dfn}

Then, one way to define a building is as follows (for the equivalence for the more common definition, see \cite[Theorem~4.131]{AB}).

\begin{dfn}[Building]
A building is a thick chamber complex together with a family of subcomplexes, called apartments, which satisfy the following axioms:
\begin{itemize}
\item Each apartment is a thin chamber complex.
\item Any two faces in the complex are contained in a common apartment.
\item Any two apartments have an isomorphism which fixes their intersection.
\end{itemize}
A building is said to be spherical/affine/hyperbolic if it is finite/locally finite/otherwise, respectively.
\end{dfn}

Let us note that if $B$ is a $d$-dimensional building (i.e. a $d$-dimensional chamber complex which satisfy the axioms of the building),
then each apartment of $B$ is also $d$-dimensional.

\begin{rem}
Throughout this paper we only concerns ourselves with buildings which are simplicial complexes.
However, it should be mentioned that buildings can also be poly-simplicial complexes (just by allowing chamber complexes to be such).
\end{rem}

Let us present an example of a spherical building.
\begin{exm}\cite[\S~4.3]{AB} \label{exm-sph}
Let $q$ be a prime power, $d \in \mathbb{N}$, and denote $V = \mathbb{F}_q^d$.
Consider the simplicial complex, $\mathbb{P}(V)$, whose vertices are the proper (i.e. not $\{0\}$ or $V$) subspaces of $V$,
and his faces are the flags of subspaces in $V$, i.e. $\{0\} < W_1 < \ldots < W_t < V$.
Then $\mathbb{P}(V)$ is a $(d-2)$-dimensional $q$-thick spherical building.
Moreover, the group $PGL_d(\mathbb{F}_q)$ acts on $\mathbb{P}(V)$ in a strongly transitive way (see below).
\end{exm}

Next, we wish to list some basic properties of spherical buildings, for the complete proofs we refer to \cite{AB}.

\begin{lem} \label{lem-build-aptsize}
For any $d\in \mathbb{N}$ define $\theta_d:= \max\{2^d \cdot (d+1)!,192\cdot 11!\}$.
Then each apartment in a $d$-dimensional spherical building is of size at most $\theta_d$.
\end{lem}

\begin{proof}
By \cite[Theorem~4.131]{AB} each apartment in a spherical building is a spherical Coxeter complex,
and by \cite[\S~1.3,1.5.6]{AB} the spherical Coxeter complexes were completely classified,
and $\theta_d$ is taken to be the maximal size of all such possible complexes.
\end{proof}

\begin{cor} \label{cor-build-size}
For any $d,q\in \mathbb{N}$ define $\theta_d:= \max\{2^d \cdot (d+1)!,192\cdot 11!\}$ and $Q_{d,q}:= ((d+1)\cdot (q+1))^{\theta_d}$.
Then each $d$-dimensional $q$-thick spherical building is of size at most $Q_{d,q}$.
\end{cor}

\begin{proof}
Let $C_0$ some chamber in the building. Since the building is $d$-dimensional and $q$-thick, we get that $C_0$ has at most $(d+1)\cdot (q+1)$ chambers of gallery distance at most $1$ from $C_0$. Iterating this fact we get that there are at most $((d+1)\cdot (q+1))^n$ chambers of gallery distance at most $n$ from $C_0$. Since each chamber in the building is contained in some apartment containing $C_0$, and since the distance of two chambers inside the apartment is at most the size of the apartment, which by Lemma \ref{lem-build-aptsize} is at most $\theta_d$, the claim is proven.
\end{proof}

\begin{rem}
The bound $((d+1)\cdot (q+1))^{\theta_d}$ on the size of the spherical building is by no mean tight, since we are not trying to optimize the constants in this work.
\end{rem}

%\begin{rem}
%Let us note right away that spherical buildings are neither sparse (i.e. of bounded degree), nor dense (i.e. of degree approximately the number of vertices).
%\end{rem}

\begin{dfn}[Type function]
A $d$-dimensional complex, $X$, is said to admit a $(d+1)$-type function on the vertices if there is a function $\tau_X : X(0) \rightarrow [d]:=\{0,1,\ldots,d\}$, such that, setting $V_i:= \tau_X^{-1}(\{i\})$ for $i=0,\ldots,d$, then $X(0) = \bigsqcup_{i=0}^d V_i$ and $X(d) \subset \prod_{i=0}^d V_i$.
\end{dfn}

\begin{lem}\cite[Proposition~4.6]{AB} \label{lem-build-type}
Let $B$ be a $d$-dimensional building. Then $B$ admits a $(d+1)$-type function on its vertices.
\end{lem}

\begin{lem}\cite[Proposition~4.9]{AB} \label{lem-build-link}
Let $B$ be a building and let $\sigma \in B$ be any face in it.
Then the link, $B_\sigma$, is also a building.
\end{lem}

\begin{lem}\cite[Proposition~4.40]{AB} \label{lem-build-gallconv}
Let $B$ be a building and let $A$ be an apartment in it.
Then $A$ is gallery convex, i.e. for any two chambers in $A$, $C,C' \in A$, and any gallery from $C$ to $C'$ in the building $B$, $C=C_0,\ldots,C_n=C'$,
which is of minimal length among all possible galleries from $C$ to $C'$, then the gallery sits completely inside the apartment $A$, i.e. $C_0,\ldots,C_n \in A$.
\end{lem}

\begin{lem}\cite[Proposition~5.122~(2)]{AB} \label{lem-build-opp}
Let $B$ be a spherical building and let $C$ be a chamber in $B$.
For any apartment containing $C$, $A$, there is a unique chamber in $A$, denoted $C^{op}_A$,
which is of maximal gallery distance from $C$.
\end{lem}

Next we define a notion of a building which posses many symmetries.
\begin{dfn}[Strongly transitive action]
A building, $B$, is said to posses a strongly transitive action, if there exist a group of automorphisms on the building, $G \leq Aut(B)$, such that:
\begin{itemize}
\item $G$ preserves the $(d+1)$-type function of the building as defined in Lemma \ref{lem-build-type}.
\item For any two pairs, $(C_1,A_1)$ and $(C_2,A_2)$, of a chamber, $C_i$, and an apartment containing the chamber, $A_i$, $i=1,2$,
there exists $g \in G$, such that $g(C_1)=C_2$ and $g(A_1)=A_2$.
\end{itemize}
\end{dfn}

\begin{lem} \label{lem-build-regular}
Let $B$ be a $d$-dimensional building and $G$ a group that acts strongly transitively on it.
Then $B$ is a regular complex (as defined in \S \ref{sec-mix}).
\end{lem}

\begin{proof}
By Lemma \ref{lem-build-type}, there exists a type-function $\tau_B : B(0) \rightarrow [d]$, i.e. if $V_i := \tau_B^{-1}(\{i\}) \subset B(0)$, $i=0,\ldots,d$, then $B(0) = \bigsqcup_{i=0}^d V_i$ and $B(d) \subset \prod_{i=0}^d V_i$.
Now, let $I \subset J \subset [d]$ be two type sets, and let $\sigma,\sigma' \in B \cap \prod_{i \in I} V_i$ be two $I$-type faces.
Choosing two chambers, $C$ and $C'$, which contains $\sigma$ and $\sigma'$ respectively, and by the second property of the strong transitivity there exists $g \in G$ such that $g(C)=C'$. Also by the first property of the strong transitivity, $g$ preserves the types of $\tau_B$,
hence $g(\sigma) = \sigma'$. Since $g$ is an automorphism, the $J$-type faces containing $\sigma$ are mapped bijectively to the
$J$-type faces containing $\sigma'$, in particular they are of the same cardinality, proving that the building is regular.
\end{proof}

Throughout this section we shall make use of the following notion from group theory.
\begin{dfn}[Stabilizer]
Let $X$ be any simplicial complex and $G \leq Aut(X)$ a group of automorphisms on $X$.
For any $\sigma \in X$, define the stabilizer of $\sigma$ in $G$ to be the following subgroup of $G$:
$$G_\sigma = stab_G(\sigma) = \{g \in G \;|\; g(\sigma) = \sigma \}.$$
\end{dfn}

\begin{lem} \label{lem-build-orb}
Let $B$ be a building and $G$ a group that acts strongly transitively on it.
Then for any face, $\sigma \in B$, and any apartment containing it, $A$,
every $G_\sigma$-orbit in $B$ passes through the apartment $A$,
i.e. for any $\tau \in B$ there exists $g \in G_\sigma$ such that $g(\tau) \in A$.
\end{lem}

\begin{proof}
Let $C \in A$ be a maximal face containing $\sigma$, and let $C' \in B$ be a maximal face containing $\tau$.
By the second axiom of the building, there exists an apartment $A'$ which contains both $C$ and $C'$.
By the strong-transitivity, there exists $g \in G$ such that $g(C) = C$ and $g(A')=A$.
On the one hand, $g$ preserves the type function, hence $g(\sigma) = \sigma$, i.e. $g \in stab_G(\sigma)$.
On the other hand, $g(A') = A$, in particular $g(\tau) \subset g(C') \in A$, as needed.
\end{proof}

\begin{lem} \label{lem-build-linksym}
Let $B$ be a building which posses a strongly transitive action, and let $\sigma \in B$.
Then the link $B_\sigma$ also posses a strongly transitive action.
\end{lem}

\begin{proof}
Let $G \leq Aut(B)$ be the group that acts strongly transitive on $B$.
define $G_\sigma = stab_G(\sigma)$ the stabilizer of $\sigma$ in $G$, then $G_\sigma$ admits an action on the link,
i.e. $G_\sigma \leq Aut(B_\sigma)$. The action of $G_\sigma$ is type-preserving since the action of $G$ is.
Any pair of a chamber and an apartment containing it inside $B_\sigma$, $(C, A)$, can be lifted to a pair of a chamber and an apartment containing it inside $B$, $(\tilde{C}, \tilde{A})$.
Let $(C, A)$ and $(C', A')$, be two pairs, where each pair contains a chamber and an apartment containing that chamber, inside $B_\sigma$.
Lifting them to such pairs in $B$, $(\tilde{C}, \tilde{A})$ and $(\tilde{C'}, \tilde{A'})$, then there is $g \in G$,
such that $(g(\tilde{C}), g(\tilde{A})) = (\tilde{C'}, \tilde{A'})$, and since both pairs contains $\sigma$,
then $g \in G_\sigma$, and hence $g(C) = C'$ and $g(A)=A'$, which finishes the proof.
\end{proof}

%%%%%%%%%%%%%%%%%%%%%%%%%%%%%%%%%%%%%%%%%%%%%%%%%
\subsection{Expansion of spherical buildings}
It was first observed by Gromov \cite{Gro} that spherical buildings are coboundary expanders (see \cite{LMM} for a simplified proof).

\begin{thm}\label{thm-build-cob}
For any $d\in \mathbb{N}$, there exist $\beta_d > 0$, such that any $d$-dimensional spherical building is a $\beta_d$-coboundary expander.
\end{thm}

The purpose of this subsection is to prove that the spherical buildings of sufficiently large thickness are good skeleton expanders.

\begin{thm} \label{thm-build-skel}
Let $X$ be a $d$-dimensional $q$-thick spherical building which posses a strongly transitively action.
Then $X$ is an $\frac{\theta_d}{\sqrt{q}}$-skeleton expander, where $\theta_d:= \max\{2^d \cdot (d+1)!,192\cdot 11!\}$.
\end{thm}

The strategy for proving Theorem \ref{thm-build-skel} is as follows:
First, we define a property for bipartite graphs (symmetric-convex), and show that such graphs have good bound on their second largest eigenvalue.
Second, we show that the type-induced graphs of a spherical building which posses a strongly transitively action satisfy this property.
Finally, by applying Corollary \ref{cor-skeleton-expansion} we get the skeleton expansion.

\begin{dfn}[Symmetric convex graph]
Let $X=(V_1\bigsqcup V_2,E)$ be a bipartite graph, let $G\leq Aut(X)$ be a group that acts by graph automorphisms on $X$, and let $\theta \in \mathbb{N}$.
For any vertex $v$ in $X$, denote by, $G_v = stab_G(v)= \{g\in G|g(v)=v\}$, its stabilizer.

Say that $X$ is a $\theta$-symmetric-convex graph, if for any $v \in V_1$, then:
\begin{enumerate}
\item \label{sc-1} The number of $G_v$-orbits in $X$ is at most $\theta$.
\item \label{sc-2} There is a unique $G_v$-orbit of vertices in $V_1$ of maximal distance from $v$,
and a unique $G_v$-orbit of vertices in $V_2$ of maximal distance from $v$.
\item \label{sc-3} For any vertex $u$, which is not of maximal distance from $v$,
the number of neighbors $w$ of $u$ such that $dist(v,w) < dist(v,u)$, is at most $\theta$.
\end{enumerate}
\end{dfn}

\begin{pro} \label{pro-symm-conv->spec}
Let $X$ be a bipartite $(k,k')$-biregular $\theta$-symmetric-convex connected graph,
then the normalized second largest eigenvalue of $X$ is bounded by,
\begin{equation}
\lambda(X) \leq \theta \cdot \max\{\frac{1}{\sqrt{k}},\frac{1}{\sqrt{k'}}\}
\end{equation}
\end{pro}

\begin{proof}
Let $A=A_X \in End(\mathbb{C}^{V(X)})$ be the adjacency operator of $X$, and let $Spec(A) = \{\lambda_n, \ldots, \lambda_2,\lambda_1\}$ be its set of eigenvalues.
Let us recall some basic facts (see \cite[\S~3]{EGL}): 
Since $X$ is an undirected graph the operator $A$ is self-adjoint, hence $Spec(A) \subset \mathbb{R}$. Since $X$ is bipartite, then $\lambda_{n-i+1} = -\lambda_i$ for any $i=1,\ldots,n$. Finally, since $X$ is $(k,k')$-biregular, then $\lambda_1 = \sqrt{k\cdot k'}$.

Note that any eigenvectors of the eigenvalues, $\lambda_n,\lambda_1$, are non-zero on every vertex,
and if $f_2 \in \mathbb{C}^{V(X)}$ is an eigenvector of $\lambda_2$, there most exist $v \in V_1$, such that $f_2(v)\ne 0$.

Fix such a vertex $v\in V_1$, let $K=stab_G(v)$ be its stabilizer in $G$, and define the following directed-multi-graph $\bar X = X/K$ as follows:
The vertices of $\bar X$ are the $K$-orbits of the vertices of $X$, i.e. $[u]=\{k(u)|k \in K\}$ for some $u \in V(X)$,
and the number of edges in $\bar X$ from $[u]$ to $[w]$ is equal to the number of edges in $X$
between the vertex $u$ and the set of vertices $\{k(w)|k\in K\}$ (note that this is independent of the choice of $u$).
Finally, let $\bar A \in End(\mathbb{C}^{V(\bar X)})$ be the adjacency operator of $\bar X$, and let $Spec(\bar A) \subset \mathbb{C}$ be its set of eigenvalues.
By the definition of $\bar X$ and $\bar A$, we get that for any $u\in V(X)$ and any $[w]\in V(\bar X)$,
\begin{equation}
\bar A_{[u],[w]} = \sum_{w\in [w]} A_{u,w}.
\end{equation}
Note that since $\bar X$ is directed, a priori there is no reason for all the eigenvalues of $\bar A$ to be real, however,
if $\lambda$ is an eigenvalue of $\bar A$, and $f \in \mathbb{C}^{V(\bar X)}$ is an eigenvector of $\lambda$,
then defining $f' \in \mathbb{C}^{V(X)}$ by $f'(u) = f([u])$, we get that for any $u \in V(X)$,
\begin{equation}
A f'(u) = \sum_{w \in V(X)} A_{u,w} f'(w) = \sum_{[w] \in V(\bar X)} \bar A_{[u],[w]} f([w]) = \bar A f([u])
= \lambda \cdot f([u]) = \lambda \cdot f'(u)
\end{equation}
hence, $f'$ is an eigenvector of $A$ with eigenvalue $\lambda$. In particular,
\begin{equation}
Spec(\bar A) \subset Spec(A) \subset \mathbb{R}
\end{equation}
On the other hand, if $\lambda \in Spec(A)$, with an eigenvector $f \in \mathbb{C}^{V(X)}$ such that $f(v) \ne 0$,
then defining $\bar f \in \mathbb{C}^{V(\bar X)}$ by $\bar f([u]) = \frac{1}{|K|}\sum_{k \in K} f(k(u))$,
we get $\bar f \not \equiv 0$, such that for any $[u] \in V(\bar X)$,
\begin{multline}
\bar A \bar f([u]) = \sum_{[w] \in V(\bar X)} \bar A_{[u],[w]} \bar f([w])
=  \sum_{[w] \in V(\bar X)} \bar A_{[u],[w]} \frac{1}{|K|}\sum_{k \in K} A f(k(w)) \\
=\frac{1}{|K|}\sum_{k \in K} \sum_{w \in V(X)} A_{u,w} f(k(w))
= \frac{1}{|K|}\sum_{k \in K} A f(k(u)) =  \frac{1}{|K|}\sum_{k \in K} \lambda \cdot f(k(u)) = \lambda \cdot \bar f([u])
\end{multline}
hence $\bar f$ is an eigenvector of $\bar A$ with eigenvalue $\lambda$. In particular,
\begin{equation}
\lambda_n,\;\lambda_2,\; \lambda_1 \in Spec(\bar A) \subset \mathbb{R}
\end{equation} \label{trace}
and by applying the trace formula for $\bar A^2$, we get
\begin{equation}
\lambda_2^2 + 2\cdot k\cdot k' = \lambda_n^2 + \lambda_2^2 + \lambda_1^2 \leq \sum_{\lambda \in Spec(\bar A)} \lambda^2 = tr(\bar A^2)
\end{equation}

Finally, let us use the properties of the symmetric-convex graph:
By property \ref{sc-1} the number of vertices in $\bar X$ is at most $\theta$.
By property \ref{sc-2} in each of the two parts of $\bar X$, there is a unique vertex, $[v_i]$, $i=1,2$, of maximal distance from $[v]$,
and since $X$ is $(k,k')$-biregular, so does $\bar X$, and hence $[v_1],[v_2]$ has at most $k\cdot k'$ directed 2-paths starting and ending with them.
By property \ref{sc-3} for any other vertex $[u]$ in $\bar X$, $[u] \ne [v_1],[v_2]$, the number of directed 2-paths starting and ending with $[u]$
is at most $\theta \cdot \max\{k,k'\}$, since such a 2-path corresponds to a a following 2-path in $X$, $u \sim w \sim u'=k(u)$ for some vertex $w$ and $k \in K$,
so either $dist(v,u) < dist(v,w)$ in which case there are at most $\theta$ such possible $u'=k(u)$ (note $dist(v,u) = dist(v,k(u))$),
or $dist(v,u) > dist(v,w)$ in which case there are at most $\theta$ such possible $w$.
All in all we get that
\begin{equation}
tr(\bar A^2) = \sum_{[u] \in \bar X} \#\{\mbox{directed 2-paths starting and ending with } [u]\} \leq 2\cdot k\cdot k' + (\theta-2) \cdot \theta \cdot \max\{k,k'\}
\end{equation}
Combining this with \eqref{trace}, we get
\begin{equation}
\lambda_2^2 \leq  tr(\bar A^2) - 2\cdot k\cdot k' \leq \theta^2 \cdot \max\{k,k'\}
\end{equation}
and noting that $\lambda(X) = \frac{\lambda_2}{\lambda_1} = \frac{\lambda_2}{\sqrt{k \cdot k'}}$, finishes the proof.
\end{proof}

The following Proposition shows that the type-induced bipartite graphs of the spherical buildings are symmetric-convex,
with a constant $\theta$ that depends only on the dimension of the building (and not on its thickness!).

\begin{pro} \label{pro-build->symm-conv} 
Let $B$ be a $d$-dimensional $q$-thick spherical building which posses a strongly transitive action.
Let $i\ne j \in [d]$ and let $B_{(i,j)}$ be the $(i,j)$-type induced bipartite biregular graph of $B$.
Then $B_{(i,j)}$ is an $\theta_d$-symmetric-convex graph with both regularity degrees at least $q+1$, where $\theta_d:= \max\{2^d \cdot (d+1)!,192\cdot 11!\}$.
\end{pro}

\begin{proof}
Let $v$ be a vertex in the building, $G_v = stab_G(v)$, and let $A$ be some fixed apartment that contains $v$.

First, let us prove the claim on the regularity degrees:
By Lemma \ref{lem-build-regular}, $B$ is a regular complex, hence $B_{(i,j)}$ is a bipartite biregular graph.
Assume $v$ is of $i$-type, so it most be contained in some  panel of $j$-cotype, $\sigma$,
(i.e. the $\sigma$ contain a vertex of each type except $j$),
and by the $q$-thickness $\sigma$ is contained in $q+1$ chambers, $C_0,\ldots,C_q$, each of which contains a unique $j$-type vertex, $u_0,\ldots,u_q$,
which is a neighbour of $v$ in the graph $B_{(i,j)}$ (and of course the same reasoning apply when replacing $i$ and $j$).

1) By Lemma \ref{lem-build-orb}, the number of $G_v$-orbits is bounded by the size of an apartment in the building,
which, by Lemma \ref{lem-build-aptsize}, is bounded by $\theta_d$.

2) Let $C$ be some chamber in $A$ which contains $v$.
By Lemma \ref{lem-build-opp} there is a unique chamber $C^{op}_A \in A$ of maximal gallery distance from $C$ in $A$,
and let $e=\{v_1,v_2\} \subset C^{op}$ be (the unique) edge of type $\{i,j\}$ inside it.
Now, since gallery distance is coarser then graph distance (any gallery path contains in it a graph path),
then $v_1$ and $v_2$ are the two farthest vertices from $v$ (w.r.t. the graph metric of $B_{(i,j)}$), of type $i$ and $j$ respectively, inside the apartment $A$.
On the other hand, since $G_v$ is a collection of automorphisms, hence preserves distances,
and by Lemma \ref{lem-build-opp}, there is a unique chamber, $C^{op}_{A'}$, of maximal gallery distance from $C$, in each apartment $A'$,
we get that for any two apartments $A'$, $A''$, there is some  $k\in G_v$ such that $k(C_{A'}^{op}) = C_{A''}^{op}$.
So, $[v_1]$ is the unique $G_v$-orbit in $V_i$ of maximal distance vertices from $v$, and similarly $[v_2]$  is the unique $G_v$-orbit in $V_j$ of maximal distance vertices from $v$.

3) Since $u$ is not of maximal distance from $v$, there is some neighbour $w^*$ of $u$ such that $dist(v,u) < dist(v,w^*)$.
Now by the axiom of the building let $A$ be an apartment that contains both $v$ and the edge $(u,w^*)$.
Let $w_1,\ldots,w_n$ be all the vertices in $B_{(i,j)}$, which are neighbours of $u$ and satisfy $dist(v,w_i) < dist(v,u)$ for any $i=1,\ldots,n$.
Then, any minimal path from $v$ to $w^*$ which passes through $u$, most also pass through some $w_i$.
Hence, since gallery distance is coarser then graph distance, any minimal gallery in the building from a maximal face containing $\{v\}$,
to a maximal face containing $\{u,w^*\}$, most also pass through a maximal face containing $\{u,w_i\}$ for some $i$.
Now, if $A$ is an apartment containing $v$ and $\{u,w^*\}$, then by Lemma \ref{lem-build-gallconv} all these minimal galleries lies in $A$,
in particular all $w_1,\ldots,w_n$ lies in $A$, i.e. $n \leq |A| \leq \theta_d$ (where the last inequality is Lemma \ref{lem-build-aptsize}).
\end{proof}

Finally, by combining Propositions \ref{pro-symm-conv->spec}, \ref{pro-build->symm-conv} and Corollary \ref{cor-skeleton-expansion}, we are able to prove that spherical buildings are good skeleton expanders (Theorem \ref{thm-build-skel}).

\begin{proof}[Proof of Theorem \ref{thm-build-skel}]
By Proposition \ref{pro-build->symm-conv}, each type induced bipartite graph, $B_{(i,j)}$, of the building is $\theta_d$-symmetric-convex.
By Proposition \ref{pro-symm-conv->spec}, each $B_{(i,j)}$ has a normalized largest non-trivial eigenvalue at most $\frac{\theta_d}{\sqrt{q}}$.
Finally applying Corollary \ref{cor-skeleton-expansion} we get that the building is an $\frac{\theta_d}{\sqrt{q}}$-skeleton expander.
\end{proof}

%%%%%%%%%%%%%%%%%%%%%%%%%%%%%%%%%%%%%%%%%%%%%%%%%%%%%%%%%%%%%%%%%%%%%%%%%%%%%%%%%%%%%%%%%%%%%%%%%%%%%%%%%%%%%%%%%%%%%%%%%%%%%%%%%%%%%%%%%%%%%%%%%%%%%%%%%%%%%%%%%%%%%%%%%
%%%%%%%%%%%%%%%%%%%%%%%%%%%%%%%%%%%%%%%%%%%%%%%%%%%%%%%%%%%%%%%%%%%%%%%%%%%%%%%%%%%%%%%%%%%%%%%%%%%%%%%%%%%%%%%%%%%%%%%%%%%%%%%%%%%%%%%%%%%%%%%%%%%%%%%%%%%%%%%%%%%%%%%%%
\section{Ramanujan complexes} \label{sec-ram}

Ramanujan complexes were defined in \cite{LSV1}, and were explicitly constructed in \cite{LSV2}.
They are finite quotients of Bruhat-Tits buildings of type $\tilde{A}$, which exhibits excellent spectral properties.
For more on Ramanujan complexes, we refer the readers for the survey \cite{Lub2}.

The object of this section is to show that Ramanujan complexes of sufficiently large degree satisfies the requirement of Theorem \ref{thm-cri}, namely:

\begin{thm} \label{thm-ram-cri}
Let $X$ be a $d$-dimensional $q$-thick Ramanujan complex. Then:
\begin{itemize}
\item The complex $X$ is $Q_{d,q}$-bounded degree, where $Q_{d,q}$ is as in Corollary \ref{cor-build-size}.
\item Each proper link of $X$ is a $\beta_d$-link coboundry expander, where $\beta_d$ is as in Theorem \ref{thm-build-cob}.
\item Each proper link of $X$ is a $\frac{\theta_d}{\sqrt{q}}$-skeleton expander, where $\theta_d$ is as in Theorem \ref{thm-build-skel}.
\item $X$ itself is a $2^d \cdot q^{-(d-1)/2}$-skeleton expander.
\end{itemize}
\end{thm}

Before proving the Theorem we shall need the following Lemma.
\begin{lem} \label{lem-ram-link}
Any proper link of a Ramanujan complex is a spherical building which posses a strongly transitive action.
\end{lem}

\begin{proof}[Proof of Lemma \ref{lem-ram-link}]
First, since a Ramanujan complex is a quotient of an affine building of type $\tilde{A}$, any link of a Ramanujan complex is a link of the covering affine building.
Second, since an $\tilde{A}$-type building posses a strongly transitive action, then by Lemmas \ref{lem-build-link} and \ref{lem-build-linksym}, each link of it is also a building which posses a strongly transitive action.
Finally, since an affine building is locally finite, then each proper link of it is finite, hence a spherical building.
\end{proof}

\begin{proof}[Proof of Theorem \ref{thm-ram-cri}]
The first three claims follows from Lemma \ref{lem-ram-link} together with Corollary \ref{cor-build-size}, Theorem \ref{thm-build-cob} and Theorem \ref{thm-build-skel}.
We are left to show that $X$ itself is an excellent skeleton expander.
Using the notation of \cite[Section~2.1]{Lub2}, the adjacency operator of the type induced graph $X_{(i,j)}$ is the Hecke operator $A_{j-i}$ restricted to that graph,
hence by \cite[Remark~2.1.5]{Lub2},
\begin{equation}
\lambda(X) \leq \max_{1\leq i \leq d} \lambda(A_i) \leq \max_{1\leq i\leq d} {d \choose i} q^{-\frac{i(d-i)}{2}} \leq 2^d q^{\frac{-(d-1)}{2}}
\end{equation}
Combining this with Corollary \ref{partite-mix}, proves that $X$ is a $(2^d \cdot q^{-(d-1)/2})$-skeleton expander.
\end{proof}

Consequentially, combining the above Theorem \ref{thm-ram-cri}, together with Theorems \ref{thm-iso}, \ref{thm-cri} and \ref{thm-Gromov}, give us the following high dimensional expansion properties for Ramanujan complexes of sufficiently large degree (this is Corollary \ref{cor-main-4} from the introduction, together with the coisoperimetric inequality for small cochains).

\begin{cor} \label{cor-ram-exp}
For any $d ,q\in \mathbb{N}$, there exists $q_0=q_0(d)>0$, $\mu=\mu(d)>0$, $\bar \epsilon=\bar \epsilon(d) >0$, $\epsilon=\epsilon(d,q)>0$ and $c=c(d,q)>0$, such that if $X$ is the $d$-dimensional skeleton of a $(d+1)$-dimensional $q$-thick Ramanujan complex with $q\geq q_0$, then: 
\begin{itemize}
\item If $A$ is a locally minimal cochain of $X$ and $\|A\| \leq \mu$, then $\|\delta(A)\| \geq \bar \epsilon\cdot \|A\|$.
\item $X$ is an $(\epsilon, \mu)$-cosystolic expander.
\item $X$ posses the $c$-topological overlapping property.
\end{itemize}
\end{cor}

Finally, by \cite[Theorem~1.1]{LSV2}, we get that for any $2 \leq d\in \mathbb{N}$ and any  prime power $q$, there exist an infinite family of $d$-dimensional $q$-thick Ramanujan complexes. Moreover, these Ramanujan complexes are constructed explicitly.
Hence, by applying Corollary \ref{cor-ram-exp} on the Ramanujan complexes constructed in \cite{LSV2}, we get an affirmative answer to Gromov's question from the introduction (Theorem \ref{thm-main-1}). 

As mentioned in the introduction, as an immediate application of Theorem \ref{thm-main-1}, we get:

\begin{cor}[Corollary \ref{cor-gromov-guth}] \label{cor-gromov-guth-2}
For any $D\geq 2d+1$, there exists $C=C(D)>0$ and an infinite family of $d$-dimensional bounded degree complexes, $\{X_n\}_n$, which satisfies the following:
For any embedding of $X_n$ into $\mathbb{R}^D$, such that the distance between the images of any two non adjacent simplices of $X_n$ is at least $1$, 
then the volume of the $1$-neighborhood of the image of $X_n$, is at least $C\cdot |X_n|^{\frac{1}{D-d}}$.
\end{cor}

\begin{proof}[Proof of Corollary \ref{cor-gromov-guth-2}]
Combine Theorem \ref{thm-main-1} with \cite[Proposition~2.5]{GG}.
\end{proof}

The original result \cite[Theorem~2.2]{GG} gave the slightly weaker result that, for any $\epsilon>0$, there exists an infinite family of $d$-dimensional complexes, $\{X_n\}_n$, whose degree might depends of $\epsilon$ (but not on $|X_n|$), and the lower bound on the volume of the $1$-neighborhood is $C\cdot |X_n|^{\frac{1}{D-d}-\epsilon}$.
The reader is referred to \cite{GG} to learn more on the work of Kolmogorov and Barzdin, and its generalization to higher dimensions.

We close this work with the following concluding remarks:

\begin{rem}
Theorem \ref{thm-ram-cri} and Corollary \ref{cor-ram-exp} holds for any finite quotient of a Bruhat-Tits building of large degree, not just for Ramanujan complexes. 
To prove this generalization, one needs only to prove that the underlying graph of any finite quotient of a Bruhat-Tits building of large degree, is a good expander graph.
This can be done by using Oh's explicit property (T) \cite{Oh}, to get explicit bounds on the second eigenvalue of such a graph.
\end{rem}

\begin{rem}
Note that for a $d$-dimensional Ramanujan complexes of sufficiently large degree,
we were only able to prove that their $(d-1)$-skeletons are cosystolic expanders.
Following \cite[Conjecture~3.2.4]{Lub2}, we suspect that the Ramanujan complexes themselves (and in fact all finite quotients of Bruhat-Tits buildings) are cosystolic expanders.
In contrast, in \cite[Proposition~1.5]{KKL} it was shown that not all Ramanujan complexes of sufficiently large degree are coboundary expanders, i.e. cosystolic expansion is the best one could hope for in general.
\end{rem}

%%%%%%%%%%%%%%%%%%%%%%%%%%%%%%%%%%%%%%%%%%%%%%%%%%%%%%%%%%%%%%%%%%%%%%%%%%%%%%%%%%%%%%%%%%%%%%%%%%%%%%%%%%%%%%%%%%%%%%%%%%%%%%%%%%%%%%%%%%%%%%%%%%%%%%%%%%%%%%%%%%%%%%%%%
%%%%%%%%%%%%%%%%%%%%%%%%%%%%%%%%%%%%%%%%%%%%%%%%%%%%%%%%%%%%%%%%%%%%%%%%%%%%%%%%%%%%%%%%%%%%%%%%%%%%%%%%%%%%%%%%%%%%%%%%%%%%%%%%%%%%%%%%%%%%%%%%%%%%%%%%%%%%%%%%%%%%%%%%%
%%%%%%%%%%%%%%%%%%%%%%%%%%%%%%%%%%%%%%%%%%%%%%%%%%%%%%%%%%%%%%%%%%%%%%%%%%%%%%%%%%%%%%%%%%%%%%%%%%%%%%%%%%%%%%%%%%%%%%%%%%%%%%%%%%%%%%%%%%%%%%%%%%%%%%%%%%%%%%%%%%%%%%%%%

\end{document}